\documentclass[a4paper, 11pt]{article}
\usepackage{amsmath, amssymb}
\usepackage{amsthm}
\newtheorem{Thm}{Theorem}[section]
\newtheorem{Def}[Thm]{Definition}
\newtheorem{lem}[Thm]{Lemma}
\newtheorem{Prop}[Thm]{Proposition}
\newtheorem{Rem}[Thm]{Remark}

\newcommand{\Ker}{\mathop{\mathrm{Ker}}\nolimits}
\newcommand{\tr}{\mathop{\mathrm{tr}}\nolimits}
\newcommand{\Ind}{\mathop{\mathrm{Ind}}\nolimits}
\newcommand{\Res}{\mathop{\mathrm{Res}}\nolimits}
\newcommand{\Hom}{\mathop{\mathrm{Hom}}\nolimits}

\begin{document}

\title{On the reduction modulo $p$ of representations \mbox{of a quaternion division algebra
over a $p$-adic field}}
\author{Kazuki Tokimoto} 
\date{}
\maketitle


\section{Introduction} \label{section:intro}
Let $F$ be a non-Archimedean local field with a finite residue field of characteristic $p$. The local Langlands correspondence (LLC) establishes a canonical bijection between the set of isomorphism classes of $n$-dimensional Frobenius semisimple Deligne representations of the Weil group $\mathcal{W} _F$ and the set of isomorphism classes of irreducible admissible representations of $\mathrm{GL} _n(F)$. Recently two kinds of analogues of the correspondence have been studied extensively, that is, the $p$-adic Langlands correspondence and the mod $p$ Langlands correspondence. The former seeks for an extended correspondence involving $p$-adic Galois representations (i.e. continuous representations of the absolute Galois group $G_F$ of $F$ over a $p$-adic field) whereas the latter looks for a correspondence between mod $p$ representations (i.e. representations over a field of characteristic $p$) and the compatibility of these two correspondences with the reduction modulo $p$ is one of the keys to their studies.

When $n=2$ and $F=\mathbb{Q} _p$, two correspondences, along with a number of satisfactory properties including the compatibility with the reduction in many cases, have been established (cf. \cite{Br3}). To motivate our study in this paper let us look at an early version of the correspondences. Breuil (cf. \cite{Br1}, \cite{Br2}) conjectured and partly proved a $p$-adic Langlands correspondence involving two-dimensional irreducible crystalline representations of $G_{\mathbb{Q} _p}$ and a mod $p$ Langlands correspondence involving two-dimensional semisimple mod $p$ representations of $G_{\mathbb{Q} _p}$. Roughly speaking, the $p$-adic Langlands correspondence was proposed by combining the local Langlands correspondence with the data necessary to obtain a $p$-adic Galois representation from a Deligne representation and the mod $p$ Langlands correspondence was discovered by classifying the mod $p$ representations of both sides by the same parameters and taking into account the compatibility with the reduction.
Our objective in this paper is an elementary consideration of a similar situation with $\mathrm{GL} _2(F)$ replaced by the multiplicative group $D^{\times}$ of a quaternion division algebra $D$ over $F$.

The local Jacquet-Langlands correspondence (LJLC) gives a canonical bijection between the set of isomorphism classes of discrete series representations of $\mathrm{GL} _2(F)$ and the set $\mathcal{A} _1(D)$ of isomorphism classes of irreducible admissible representations of $D^{\times}$. Composing with LLC, we obtain a canonical bijection between the set of isomorphism classes of two-dimensional indecomposable Deligne representations of $\mathcal{W} _F$ and the set $\mathcal{A} _1(D)$. In this paper we first classify certain irreducible mod $p$ representations of $\mathcal{W} _F$ and of $D^{\times}$, thereby proposing a mod $p$ correspondence (cf. section \ref{section:modp}), and then compute (the semisimplification of) the reduction of certain representations of $\mathcal{W} _F$ and of $D^{\times}$ (cf. section \ref{section:reduction}) to examine if this correspondence is compatible with the composite of LLC and LJLC (cf. section \ref{section:compatibility}).

Irreducible representations of $D^{\times}$ are easier to deal with than those of $\mathrm{GL} _2(F)$ in many aspects. Since $D^{\times}$ is compact modulo center, representations we treat are all finite-dimensional. Also, as the pro-$p$-subgroup $U_D^1$ is normal not only in the unit group $U_D$, but also in $D^{\times}$, every irreducible mod $p$ representation of $D^{\times}$ is inflated from that of $D^{\times}/U_D^1$, which makes irreducible mod $p$ representations of $D^{\times}$ much more accessible than those of $\mathrm{GL} _2(F)$. It follows that the groups essentially involved are isomorphic and a mod $p$ correspondence is obtained in quite a natural manner. However, it turned out that this mod $p$ correspondence and the composed correspondence are not compatible with the reduction modulo $p$ except for the simplest case of ``level zero'' (the mod $p$ correspondence here and its compatibility with the reduction in the level zero case have already been treated by \cite{Vi2}). In fact, in most cases the reduction of an irreducible representation of $D^{\times}$ contains every irreducible mod $p$ representation satisfying the obvious necessary condition (i.e. having the suitable central character). This may be natural in that every irreducible mod $p$ representation of $D^{\times}$ has a non-zero vector fixed by $U_D^1$ and the same condition forces an irreducible admissible representation of $D^{\times}$ to be of level zero. 
At least we make some observation on this phenomenon (cf. \ref{Rem:interpretation}).

\subsection*{Acknowledgements}
The author wishes to thank Christophe Breuil for listening to the author at the poster session of a conference and showing a paper \cite{BD} in preparation at the time whose appendix contains a similar computation and Yoichi Mieda for pointing out a silly misunderstanding of the author.
Also he would like to express his sincere gratitude to his family for always encouraging him,
to his friends for making him feel relaxed and motivating him,
and last but not least to his advisor Takeshi Tsuji for helpful advice, illuminating discussions and uplifting words.

The author was supported by the Program for Leading Graduate Schools, MEXT, Japan.

\subsection*{Notation and terminology} \label{no,te}
We list here the notation and the terminology which we use freely without recalling the definitions. For more on the materials considered here, see \cite{BH}, \cite{Se2} and \cite{Re} among many other standard references.

Let $F$ be a non-Archimedean local field with a finite residue field $k_F$ of characteristic $p$, over which we work in the entire paper. Let $q$ be the cardinality of $k_F$. Let $D$ denote a quaternion division algebra (i.e. a central division algebra of dimension four) over $F$.

If $E/F$ is a finite field extension, we denote the normalized discrete valuation of $E$ by $v_E:E^{\times}\rightarrow \mathbb{Z}$, the valuation ring of $E$ by $\mathfrak{o} _E$, the maximal ideal of $\mathfrak{o} _E$ by $\mathfrak{p} _E$ and the residue field by $k_E$. The unit group is denoted by $U_E=\mathfrak{o} _E^{\times}$ and its congruence subgroups by $U_E^n=1+\mathfrak{p} _E^n$ for $n\geq 1$. We denote by $\mu _E$ the group of roots of unity \emph{of order prime to $p$} in $E^{\times}$ (canonically isomorphic to $k_E^{\times}$). When we take a uniformizer of $E$, it is usually denoted by $\varpi _E$. The ramification index of the extension $E/F$ is denoted by $e(E|F)$.

We use an analogous notation for the division algebra $D$. Let us denote the reduced trace and the reduced norm of $D$ by $\mathrm{Trd}:D\rightarrow F$ and $\mathrm{Nrd}:D^{\times}\rightarrow F^{\times}$. We set $v_D=v_F\circ \mathrm{Nrd} :D^{\times}\rightarrow \mathbb{Z}$. It is a surjective homomorphism. Let $\mathfrak{o} _D=\{ x\in D^{\times}\mid v_D(x)\geq 0\} \cup \{ 0\}$ and $\mathfrak{q}=\{ x\in D^{\times}\mid v_D(x)\geq 1\} \cup \{ 0\}$denote the maximal order of $D$ and its maximal two-sided ideal. The residue ring $k_D=\mathfrak{o} _D/\mathfrak{q}$ is a finite field with $q^2$ elements. We set $U_D=\mathfrak{o} _D^{\times}=\Ker v_D$ and $U_D^n=1+\mathfrak{q} ^n$ for $n\geq 1$. If $\varpi _D \in D^{\times}$ satisfies $v_D(\varpi _D)=1$, it acts on $k_D$ by conjugation as a $q$-th power (i.e. $\varpi _Dx\varpi _D^{-1}\equiv x^q\pmod{U_D^1}$ for all $x\in U_D$).

If $E/F$ is a finite extension, $\mathcal{W} _E$ denotes the Weil group of $E$. If we take a Frobenius element $\mathrm{Fr} \in \mathcal{W} _E$, then $\mathcal{W} _E$ is the semidirect product $\mathcal{I} _E\rtimes \mathrm{Fr} ^{\mathbb{Z}}$ of the inertia subgroup $\mathcal{I} _E$ and the discrete infinite cyclic subgroup $\mathrm{Fr} ^{\mathbb{Z}}$. The unique pro-$p$-Sylow subgroup $\mathcal{P} _E$ of $\mathcal{I} _E$ is called the wild inertia subgroup and the tame inertia group $\mathcal{I} _E/\mathcal{P} _E$ is canonically isomorphic to $\varprojlim _{n} \mathbb{F} _{p^n}^{\times}$, where $\mathbb{F} _{p^n}$ denotes the field with $p^n$ elements. Through this isomorphism the conjugation of the tame inertia group by a Frobenius element $\mathrm{Fr}$ amounts to the usual action of $\mathrm{Fr}$ on $\varprojlim _{n} \mathbb{F} _{p^n}^{\times}$ (i.e. if $i:\mathcal{I} _E/\mathcal{P} _E\rightarrow \varprojlim _{n} \mathbb{F} _{p^n}^{\times}$ stands for the canonical isomorphism, we have $i((\mathrm{Fr}) g (\mathrm{Fr}) ^{-1})=\mathrm{Fr} (i(g))$). The local class field theory provides a natural homomorphism $\mathbf{a} _E:\mathcal{W} _E\rightarrow E^{\times}$ (Artin reciprocity map) which induces an isomorphism between the abelianization $\mathcal{W} _E^{\text{ab}}$ and $E^{\times}$. We normalize it so that geometric Frobenius elements are sent to uniformizers.

From now on, \emph{we fix a field isomorphism $\mathbb{C}\simeq \overline{\mathbb{Q}}_p$} and the representations over $\mathbb{C}$ are invariably identified with the representations over $\overline{\mathbb{Q}} _p$ through this isomorphism. In this paper, we consider admissible representations of locally profinite groups over $\mathbb{C}$ (or $\overline{\mathbb{Q}} _p$) and over $\overline{\mathbb{F}} _p$. To specify it the latter will often be called ``mod $p$ representations'' although it will be usually clear from the context which we are dealing with. One-dimensional smooth representations (i.e. continuous homomorphisms to $\mathbb{C} ^{\times}$ or $\overline{\mathbb{F}} _p^{\times}$) are often called characters.

Let $E/F$ be a finite extension. If $\psi$ is a character of $E$ (resp. of $D$), the level of $\psi$ is defined to be the least integer $n$ such that $\mathfrak{p} _E^n \subset \Ker \psi$ (resp. $\mathfrak{q} ^n\subset \Ker \psi$). If $\chi$ is a character of $E^{\times}$, the level of $\chi$ is defined to be the least integer $n$ such that $U_E^{n+1}\subset \Ker \chi$. Similarly, if $\Pi$ is an irreducible representation of $D^{\times}$ of dimension greater than one, the level of $\Pi$ is defined to be the least integer $n$ such that $U_D^{n+1}\subset \Ker \Pi$.



If $G$ is a locally profinite group and $H$ is a closed subgroup, $\Ind _H^G\sigma$ denotes the smooth induction of a smooth representation $\sigma$ of $H$ to $G$. (However, $H$ is always of finite index in $G$ in this paper and hence the notion of smooth induction and that of compact induction coincide.) Also, $\Res _H^G \sigma$ and $\sigma |_H$ denotes the restricted representation of a smooth representation $\sigma$ of $G$ to $H$.

\section{Classification of mod $p$ representations} \label{section:modp}
\subsection{Irreducible mod $p$ representations of $D^\times$}
In this subsection we classify irreducible admissible mod $p$ representations of $D^\times$. As the abelianization $D^{\times} \rightarrow (D^{\times})^{\text{ab}}$ is isomorphic to the reduced norm map $D^{\times} \overset{\mathrm{Nrd}}{\rightarrow} F^{\times}$, mod $p$ characters of $D^{\times}$ are in a natural one-to-one correspondence with mod $p$ characters of $F^{\times}$ (by composing $\mathrm{Nrd} :D^{\times} \rightarrow F^{\times}$).
We denote the set of isomorphism classes of irreducible admissible representations of $D^{\times}$ over $\overline{\mathbb{F}}_p$ of dimension greater than one by $\mathcal{A} _1^0(D, \overline{\mathbb{F}}_p)$ .

We begin with introducing the language which we use throughout the paper.
 \begin{Def}
 Let $K$ be an algebraically closed field.
 
 An admissible pair $(E/F, \chi)$ with values in $K$ is a pair in which $E/F$ is a tamely ramified quadratic extension and $\chi$ is a character of $E^{\times}$ with values in $K$, satisfying the following conditions:
 \begin{enumerate}
  \item $\chi$ does not factor through the norm map $\mathrm{N}_{E/F} :E^{\times} \rightarrow F^{\times}$ $($i.e. there does not exist any character $\varphi$ of $F^{\times}$ such that $\chi =\varphi \circ \mathrm{N} _{E/F})$.
  \item if $\chi |_{U_E^1}$ does factor through the norm map $\mathrm{N}_{E/F}$, then $E$ is unramified over $F$.
 \end{enumerate}
 The level of an admissible pair $(E/F, \chi)$ is defined to be the level of $\chi$. An admissible pair $(E/F, \chi)$ is said to be minimal if the level of $(\varphi \circ \mathrm{N} _{E/F}) \otimes \chi$ is not strictly smaller than that of $\chi$ for any character $\varphi$ of $F^{\times}$. We often call minimal admissible pairs simply by minimal pairs.
 
 We say that admissible pairs $(E/F, \chi)$ and $(E^{\prime} /F, \chi ^{\prime})$ are $F$-isomorphic if there exists an $F$-isomorphism $j:E\rightarrow E^{\prime}$ such that $\chi =\chi ^{\prime} \circ j$. The set of $F$-isomorphism classes of admissible pairs with values in $K$ is denoted by $\mathbb{P} _2(F, K)$. If $K=\mathbb{C}$, an admissible pair with values in $\mathbb{C}$ is simply called an admissible pair and we set $\mathbb{P} _2(F)=\mathbb{P} _2(F, \mathbb{C})$.
 \end{Def}
 \begin{Rem}
    \begin{enumerate}
   \item Since every character of a subgroup of $F^{\times}$ of index two with values in $K$ can be extended to a character of $F^{\times}$, Hilbert 90 shows that the condition \text{1} in the definition is equivalent to
   \[ \chi \neq \chi ^{\theta} \text{ for } \theta \in \mathrm{Gal} (E/F)\setminus \{ 1\}. \]
   \item Admissible pairs $(E/F, \chi)$ and $(E/F, \chi ^{\prime})$ are $F$-isomorphic if and only if $\chi ^{\theta} =\chi ^{\prime}$ for some element $\theta \in \mathrm{Gal} (E/F)$.
   \item If $\chi$ in an admissible pair $(E/F, \chi)$ is trivial on $U_E^1$ $($e.g. if the characteristic of $K$ is $p)$, $E$ is unramified over $F$ by the condition 2.
  \end{enumerate}
 \end{Rem}
 We also need another closely related notion to facilitate the computation of the reduction.
 \begin{Def} \label{Def:reg}
 Let $E/F$ be a quadratic Galois extension. Let $\xi$ be a mod $p$ character of $E^{\times}$.
 
 We say that $\xi$ is regular if $\xi$ does not factor through the norm map $\mathrm{N}_{E/F} :E^{\times} \rightarrow F^{\times}$.
 \end{Def} 
 \begin{Rem}\label{Rem:regchar}
  \begin{enumerate}
   \item If $E/F$ is an unramified quadratic extension, we have a natural identification
    \[ \mathbb{P} _2(F, \overline{\mathbb{F}} _p )= \{ \xi \mid \xi \text{ is a regular mod $p$ character of $E^{\times}$} \} /\sim , \]
           where $\sim$ is an equivalence relation defined by
    \[ \xi \sim \xi ^{\prime} \text{ if and only if } \xi ^{\theta}=\xi ^{\prime} \]
           for some element $\theta \in \mathrm{Gal} (E/F)$.
   \item Let $E/F$ and $\xi$ be as in the Definition \ref{Def:reg}.
   
           If $E/F$ is unramified, then $\xi$ is regular if and only if $\xi (\zeta_E^{q-1}) \neq 1$, where $\zeta_E \in E^{\times}$ is a generator of $\mu_E$.
           
           If $E/F$ is totally ramified, then $\xi$ is regular if and only if $\xi (-1) \neq 1$. $($Note that $\xi$ is trivial on $U_E^1$.$)$ In particular, $\xi$ is irregular if $p=2$.
   \item As in the previous remark, we often take some elements to state the result explicitly. Therefore we fix the notation here to avoid explaining the same setting repetitively.
    \begin{enumerate}
     \item \label{unramnotation} If a quadratic unramified extension $E/F$ (sometimes denoted by $E_0/F$) is fixed in the context, then we take a generator $\zeta_E \in \mu_E$ and a uniformizer $\varpi_F \in F$, and set $\zeta_F =\zeta_E ^{q+1} \in \mu_F$.
     \item \label{tameramnotation} If a \emph{tamely} quadratic totally ramified extension $E/F$ is fixed in the context, then we take a generator $\zeta_F \in \mu_F$ and a uniformizer $\varpi_E \in E$ such that $\varpi_E^2$ is a uniformizer of $F$, and set $\varpi_F =\varpi_E^2 \in F$.
    \end{enumerate}
   \end{enumerate}
  \end{Rem}
 In what follows, we mainly use the more flexible notion of regular mod $p$ characters instead of that of admissible pairs with values in $\overline{\mathbb{F}} _p$.
 \begin{Prop} \label{Prop:modquaternion}
 Let $E/F$ be an unramified quadratic extension and $\xi$ a mod $p$ character of $E^{\times}$. Let us take an $F$-embedding $E\rightarrow D$. Let us denote by $\xi ^{\dagger}$ the mod $p$ character of $E^{\times}U_D^1=F^{\times}U_D$ such that $\xi ^{\dagger}|_{E^{\times}}=\xi$ and $\xi ^{\dagger}|_{U_D^1}=1$. Finally, let us set
 \[ \pi _{\xi}=\Ind_{E^{\times}U_D^1}^{D^{\times}} \xi ^{\dagger} .\]
 Then the following assertions hold.
  \begin{enumerate}
   \item The representation $\pi _{\xi}$ is independent up to isomorphism of the choice of $F$-embedding $E\rightarrow D$.
   \item The representation $\pi _{\xi}$ is irreducible if and only if $\xi$ is regular.
    Every irreducible mod $p$ representation of $D^{\times}$ of dimension greater than one can be expressed this way. If $\xi$ and $\xi ^{\prime}$ are regular mod $p$ characters, $\pi _{\xi}$ and $\pi _{\xi ^{\prime}}$ are isomorphic if and only if $\xi ^{\theta} =\xi ^{\prime}$ for some $\theta \in \mathrm{Gal} (E/F)$.
    
    In other words, the construction above induces a bijection
    \begin{equation} \label{modparam}
    \mathbb{P} _2 (F, \overline{\mathbb{F}} _p ) \simeq \mathcal{A} _1^0(D, \overline{\mathbb{F}} _p ).
    \end{equation}
    
    In particular, irreducible admissible mod $p$ representations of $D^{\times}$ are either one-dimensional or two-dimensional. 
   \item Let $\xi$ be irregular. Then $\xi ^{\dagger}$ admits an extension to a mod $p$ character of $D^{\times}$. There exist two such extensions $\tilde{\xi} _1, \tilde{\xi} _2$ if $p\neq 2$ and only one extension $\tilde{\xi}$ if $p=2$. We have
   \begin{alignat*}{2}
        &\pi _{\xi} \cong \tilde{\xi}_1 \oplus \tilde{\xi}_2& \qquad &\text{ if } p\neq 2, \\
        &\pi _{\xi} \text{ is a non-split extension of $\tilde{\xi}$ with itself}& &\text{ if } p=2.
   \end{alignat*}
  \end{enumerate}
 \end{Prop}
 \begin{proof}
 Changing the choice of $F$-embedding only amounts to the composition of an inner automorphism of $D^{\times}$ by Skolem-Noether theorem and hence 1 follows immediately.
 
 Since $\pi _{\xi}$ is two-dimensional, it is reducible if and only if it has a one-dimensional subrepresentation, which is equivalent to the existence of an extension of $\xi ^{\dagger}$ to $D^{\times}$ by Frobenius reciprocity. One can verify by an elementary calculation that $\xi ^{\dagger}$ extends to $D^{\times}$ if and only if $\xi ^{\dagger}$ is fixed under the conjugation action of $\varpi _{D}^{\mathbb{Z}}$, where $\varpi _D \in D^{\times}$ satisfies $v_D(\varpi _D)=1$. 
 Observing that the mod $p$ character $\xi ^{\dagger}$ is trivial on the normal pro-$p$-subgroup $U_D^1$ and $\varpi _{D} \zeta _E \varpi _{D}^{-1} \equiv \zeta _E^q \pmod{U_D^1}$ for any generator $\zeta _E \in \mu _E$, we see that this is certainly equivalent to $\xi$ being irregular. If $\xi$ is irregular, any extensions of $\xi ^{\dagger}$ appear both as a subrepresentation and as a quotient representation of $\pi _{\xi}$.

 If $p\neq 2$ and $\xi$ is irregular, $\pi _{\xi}$ has two distinct subrepresentations $\tilde{\xi} _1$ and $\tilde{\xi} _2$ and is therefore a direct sum of these subrepresentations.

 Mackey's decomposition \cite[7.3]{Se1}, \cite[I.5.5]{Vi1} gives $\Res_{E^{\times}U_D^1}^{D^{\times}} \pi _{\xi} \cong \xi ^{\dagger} \oplus (\xi ^{\theta})^{\dagger}$ and Frobenius reciprocity yields
 \[ \Hom_{D^{\times}}(\pi _{\xi ^{\prime}}, \pi _{\xi}) \cong \Hom_{E^{\times}U_D^1}({\xi ^{\prime}}^{\dagger}, \xi ^{\dagger} \oplus (\xi ^{\theta})^{\dagger}) ,\]
 for any mod $p$ characters $\xi$ and $\xi ^{\prime}$ of $E^{\times}$.
 If $\xi$ is regular, this implies the stated condition for $\pi _{\xi} \cong \pi _{\xi ^{\prime}}$. Taking $\xi =\xi^{\prime}$ to be irregular, we see that $\mathrm{End}_{D^{\times}}(\pi _{\xi})$ is two-dimensional and hence $\pi _{\xi}$ is not a direct sum of two isomorphic representations if $p=2$ and $\xi$ is irregular.

 It now remains to show that the induced map (\ref{modparam}) is surjective. Let $(\pi, V)$ be an irreducible mod $p$ representation of $D^{\times}$ of dimension greater than one. As $U_D^1$ is pro-$p$ and $V$ is a direct limit of $p$-groups, a standard argument shows that the invariant part $V^{U_D^1}\neq 0$.\footnote{In fact, $U_D^1$ is normal in $D^{\times}$ and hence $V^{U_D^1}=V$ by irreducibility. In other words, every irreducible mod $p$ representation of $D^{\times}$ is inflated from that of $D^{\times}/U_D^1$. However, we are not going to need this fact in what follows.} Since the order of $U_D/U_D^1$ is prime to $p$, $V^{U_D^1}$ is semisimple as a representation of $U_D$ and in particular contains a character of $U_D$. On the other hand, $\pi$ admits a central character by Schur's lemma. These characters are consistent and define a character $\xi ^{\prime}$ of $F^{\times}U_D =E^{\times}U_D^1$ contained in $V$. Let us put $\xi =\xi ^{\prime}|_E$. By Frobenius reciprocity the inclusion of $\xi ^{\prime}$ into $V$ induces a $D^{\times}$-equivariant map $\pi _{\xi} =\Ind_{F^{\times}U_D}^{D^{\times}} \xi ^{\prime} \rightarrow \pi$, which is non-zero and hence surjective by irreducibility. Since $\pi$ is not one-dimensional, it is an isomorphism. This completes the proof.
  
 \end{proof}
 \begin{Rem} \label{Rem:quaternionunram}
 If $\xi$ is irregular $($and $E/F$ is unramified$)$, then any extension of $\xi$ is expressed as $(\varphi \circ \mathrm{Nrd})$ with some mod $p$ character $\varphi$ of $F^{\times}$ such that
  \[ \varphi (\zeta _F) =\xi (\zeta _E) \text{ and } \varphi (\varpi _F)^2=\xi (\varpi _F), \]
  where $\zeta_E, \zeta_F \text{ and } \varpi _F$ are taken as in \emph{Remark }$\ref{Rem:regchar}.\ref{unramnotation}$.
 \end{Rem}
  
 Although this much suffices for the classification of irreducible mod $p$ representations of $D^{\times}$, we study further the analogous situation with $E$ being ramified over $F$ for the convenience of computations later.
 
 \begin{Prop} \label{Prop:quaternionmodpram}
  Let $E/F$ be a totally ramified quadratic extension of $F$ and $\nu$ a mod $p$ character of $E^{\times}$. Let us take an $F$-embedding $E\rightarrow D$. Let us denote by $\nu ^{\ddagger}$ the extension of $\nu$ to $E^{\times}U_D^1$ such that $\nu ^{\ddagger}|_{U_D^1}=1$.
  
  Then the induced representation $\Ind_{E^{\times}U_D^1}^{D^{\times}} \nu ^{\ddagger}$ is independent up to isomorphism of the choice of $F$-embedding and we have
  \[ \Ind_{E^{\times}U_D^1}^{D^{\times}} \nu ^{\ddagger} \cong (\bigoplus _{\pi \in I_1} \pi ) \oplus (\bigoplus _{\tilde{\nu} \in I_2} \tilde{\nu}), \]
  where $I_1$ denotes the set of $($isomorphism classes of$)$ two-dimensional irreducible mod $p$ representations of $D^{\times}$ with a central character $\nu |_{F^{\times}}$ and $I_2$ denotes the set of mod $p$ characters of $D^{\times}$ extending $\nu ^{\ddagger}$.
  
  If $E_0$ denotes an unramified quadratic extension of $F$, we have
  \[ I_1= \{ \pi _{\xi} \mid \xi \text{ is a regular mod $p$ character of $E_0^{\times}$ extending } \nu |_{F^{\times}} \}, \]
  \[ I_2= \{ \varphi \circ \mathrm{Nrd} \mid \varphi \text{ is a mod $p$ character of $F^{\times}$ such that } (\varphi \circ \mathrm{Nrd}) |_{E^{\times}} =\nu \}, \]
  and
  \begin{equation*}
    (\# I_1, \ \# I_2)= \begin{cases}
                         (\frac{q+1}{2}, \ 0) & \text{ if } \nu \text{ is regular}\\
                         (\frac{q-1}{2}, \ 2) & \text{ if } \nu (-1)=1 \text{ and } p\neq 2\\
                         (\frac{q}{2}, \ 1) & \text{ if } p=2.
                        \end{cases}
  \end{equation*}
 \end{Prop}
 \begin{proof}
  It can be seen in exactly the same way as in the preceding proposition that the induced representation is independent up to isomorphism of the choice of $F$-embedding.

  Let us fix an $F$-embedding $E_0\rightarrow D$. Mackey's decomposition gives
  \[\Res_{E_0^{\times}U_D^1}^{D^{\times}}\Ind_{E^{\times}U_D^1}^{D^{\times}} \nu ^{\ddagger} \cong \Ind_{F^{\times}U_D^1}^{E_0^{\times}U_D^1} (\nu ^{\ddagger} |_{F^{\times}U_D^1}).\]
  Frobenius reciprocity then induces non-zero $D^{\times}$-equivariant maps $\pi _{\xi} \rightarrow \Ind_{E^{\times}U_D^1}^{D^{\times}} \nu ^{\ddagger}$ for any mod $p$ characters $\xi$ of $E_0^{\times}$ extending $\nu |_{F^{\times}}$. Of all such extensions $\xi$, regular characters come in pairs and give rise to isomorphic two-dimensional irreducible representations.

  Let us take $\zeta_{E_0} \in \mu_{E_0}$ and $\zeta_F \in \mu_F$ as in Remark \ref{Rem:regchar}.\ref{unramnotation}. If a mod $p$ character $\xi$ of $E_0^{\times}$ is an extension of $\nu |_{F^{\times}}$, we have $\xi (\zeta _{E_0})^{q+1}=\nu (\zeta _F)$, whereas $\xi$ is irregular if and only if $\xi (\zeta _{E_0}^q)=\xi (\zeta _{E_0})$. It follows that $\nu |_{F^{\times}}$ admits an irregular extension to $E_0^{\times}$ if and only if $\nu (-1)=1$ (i.e. $\nu$ is irregular) and if there are any such extensions
  \begin{alignat*}{2}
   &\text{there are two} & \qquad & \text{ if } p\neq 2 ,\\
   &\text{and there is only one} & \qquad & \text{ if } p=2.
  \end{alignat*}
  For such an irregular character $\xi$ of $E_0^{\times}$, the character $\xi ^{\dagger}$ of $E_0^{\times}U_D^1$ extends to at most two mod $p$ characters of $D^{\times}$ (two if $p\neq 2$ and one if $p=2$). Among these characters, only the character extending $\nu ^{\ddagger}$ occurs in $\Ind_{E^{\times}U_D^1}^{D^{\times}} \nu ^{\ddagger}$.
  
  Summing up dimensions, we obtain the required decomposition of the induced representation $\Ind_{E^{\times}U_D^1}^{D^{\times}} \nu ^{\ddagger}$ and the other assertions now follow easily.
 \end{proof}
 




 \begin{Rem} \label{Rem:quaterniontameram}
  Let $E$, $E_0$, $\nu$ and $F$-embeddings $E\rightarrow D$, $E_0\rightarrow D$ be as above.
  \begin{enumerate}
   \item For a mod $p$ character $\xi$ of $E_0^{\times}$ to be an extension of $\nu |_{F^{\times}}$, it is necessary and sufficient that
         \[ \xi (\zeta _{E_0})^{q+1}=\nu (\zeta _F) \text{ and } \xi (\varpi _F)=\nu (\varpi _F) \]
         and an extension $\xi$ is regular if and only if
         \[ \xi (\zeta _{E_0})^2 \neq \nu (\zeta _F). \]
         with $\zeta_{E_0}, \zeta_{F} \text{ and } \varpi_F$ as in \emph{Remark }$\ref{Rem:regchar}.\ref{unramnotation}$.
   \item If $\nu$ is irregular, then any extension of $\nu$ to $D^{\times}$ is expressed as $(\varphi \circ \mathrm{Nrd})$ with some mod $p$ character $\varphi$ of $F^{\times}$. If $p\neq 2$, then $\varphi$ satisfies
         \[ \varphi (\zeta _F)^2=\nu (\zeta _F) \text{ and } \varphi (-\varpi _F)=\nu (\varpi _E).\]
         with $\zeta_F, \varpi_E \text{ and } \varpi_F$ as in \emph{Remark }$\ref{Rem:regchar}.\ref{tameramnotation}$, and if $p=2$, then $\varphi$ is the unique mod $p$ character such that $\varphi ^2=\nu |_{F^{\times}}$.
  \end{enumerate}
 \end{Rem}
 
\subsection{Two-dimensional semisimple mod $p$ representations of $\mathcal{W} _F$}
 In order to establish a mod $p$ correspondence, we classify two-dimensional semisimple smooth mod $p$ representations of $\mathcal{W} _F$ in this subsection. Very much analogous to what holds in $D^{\times}$, one-dimensional mod $p$ representations of $\mathcal{W} _F$ are parametrized by mod $p$ characters of $F^{\times}$ by simply composing the Artin reciprocity map $\mathbf{a} _F :\mathcal{W} _F\rightarrow F^{\times}$. We denote the set of isomorphism classes of two-dimensional irreducible smooth  representations of $\mathcal{W} _F$ over $\overline{\mathbb{F}}_p$ by $\mathcal{G} _2^0(F, \overline{\mathbb{F}}_p)$.
 
 We are going to take a look at two-dimensional induced mod $p$ representations of $\mathcal{W} _F$ and in particular parametrize elements of $\mathcal{G} _2^0(F, \overline{\mathbb{F}}_p)$. For the most part, the proofs are quite similar to Proposition \ref{Prop:modquaternion}. However, observe the differences in the assertions when $E/F$ is ramified.
 \begin{Prop} \label{Prop:modGalois}
  Let $E/F$ be a quadratic Galois extension of $F$ and $\xi$ a mod $p$ character of $E^{\times}$.
  \begin{enumerate}
   \item The induced representation $\rho _{\xi} =\Ind_{\mathcal{W} _E}^{\mathcal{W} _F} (\xi \circ \mathbf{a} _E )$ is irreducible if and only if $\xi$ is regular. If $\xi$ and $\xi ^{\prime}$ are regular mod $p$ characters of $E^{\times}$, then $\rho _{\xi} \cong \rho_ {\xi ^{\prime}}$ holds if and only if $\xi ^{\theta} =\xi ^{\prime}$ for some $\theta \in \mathrm{Gal} (E/F)$. Any two-dimensional irreducible mod $p$ representations of $\mathcal{W} _F$ arise this way \emph{with $E/F$ unramified}.
   
         In particular, this construction induces a bijection
         \begin{equation} \label{modGaloisparam}
         \mathbb{P} _2 (F, \overline{\mathbb{F}} _p ) \simeq \mathcal{G} _2^0(F, \overline{\mathbb{F}} _p).
         \end{equation}
   \item Let $\xi$ be irregular. Then $(\xi \circ \mathbf{a} _E)$ admits an extension to a mod $p$ character of $\mathcal{W} _F$. There are two such extensions $\lambda _1, \lambda _2$ if $p\neq 2$ and only one extension $\lambda$ if $p=2$. We have
         \begin{alignat*}{2}
        &\rho _{\xi} \cong \lambda _1 \oplus \lambda _2& \qquad &\text{ if } p\neq 2, \\
        &\rho _{\xi} \text{ is a non-split extension of $\lambda$ with itself}& &\text{ if } p=2.
         \end{alignat*}
  \end{enumerate}
 \end{Prop}
 \begin{proof}
  As the induced representation $\rho _{\xi}$ is two-dimensional, it is reducible if and only if it has a one-dimensional subrepresentation, which is then, by Frobenius reciprocity and the functoriality of the Artin reciprocity map, equivalent to $\xi$ factoring through $\mathrm{N} _{E/F} :E^{\times} \rightarrow F^{\times}$, i.e. being irregular.
  
  If $\rho _{\xi}$ does contain one-dimensional subrepresentation, $(\xi \circ \mathbf{a} _E)$ admits two extensions $\lambda _1, \lambda _2$ to a mod $p$ character of $\mathcal{W} _F$ if $p\neq2$ and only one extension $\lambda$ if $p=2$. In any case, Frobenius reciprocity shows that $\rho _{\xi}$ has every extension of $(\xi \circ \mathbf{a} _E)$ as a subrepresentation and a quotient representation.
  
  If $p\neq2$, the two-dimensional representation $\rho _{\xi}$ contains two distinct characters $\lambda _1, \lambda _2$ and hence is isomorphic to the direct sum $\lambda _1 \oplus \lambda _2$.
  
  Frobenius reciprocity yields
  \[ \Hom _{\mathcal{W} _F} (\rho _{\xi ^{\prime}}, \rho _{\xi}) \cong \Hom _{\mathcal{W} _E} (\xi ^{\prime} \circ \mathbf{a} _E, (\xi \circ \mathbf{a} _E) \oplus (\xi ^{\theta} \circ \mathbf{a} _E)). \]
  Taking $\xi$ and $\xi ^{\prime}$ to be regular, we obtain the required condition for $\rho _{\xi} \cong \rho _{\xi ^{\prime}}$ when $\xi$ is regular. Taking $\xi =\xi ^{\prime}$ to be irregular, we see that $\mathrm{End} _{\mathcal{W} _F} (\rho _{\xi})$ is two-dimensional and hence $\rho _{\xi}$ is a non-split extension if $p=2$ and $\xi$ is irregular.
  
  Now it only remains to show that the induced map (\ref{modGaloisparam}) is surjective. Let $(\rho, V)$ be a two-dimensional irreducible mod $p$ representation of $\mathcal{W} _F$ and let $E_0/F$ be an unramified quadratic extension. As the wild inertia subgroup $\mathcal{P} _F$ is a normal pro-$p$-subgroup of $\mathcal{W} _F$, it acts trivially on $V$.\footnote{Thus any irreducible mod $p$ representation of $\mathcal{W} _F$ is inflated from a representation of the semidirect product of the tame inertia group $\mathcal{I} _F/\mathcal{P} _F$ and an infinite cyclic group generated by a Frobenius element. Irreducible representations of a finite group which is a semidirect product with the normal subgroup abelian are well-understood. See, for instance, \cite[8.2]{Se1}. The argument here can be considered as a modification.} Since the tame inertia group $\mathcal{I} _F/\mathcal{P} _F$ is abelian, profinite and of pro-order prime to $p$, the restricted representation $\rho |_{\mathcal{I} _F}$ is a direct sum of two mod $p$ characters. These two characters are distinct. Indeed, if they were identical, $\mathcal{I} _F$ would act on $V$ as scalar operators and
an eigenspace of the action of a Frobenius element would give a proper subrepresentation, which is a contradiction.

 Let $V=V_1\oplus V_2$ be the decomposition of $\rho |_{\mathcal{I} _F}$ into irreducible representations. Let us take a Frobenius element $\mathrm{Fr} \in \mathcal{W} _F$. The subgroup $(\mathrm{Fr}) ^{\mathbb{Z}}$ permutes the set $\{ V_1, V_2\}$ transitively. The stabilizer of $V_1$ with respect to this action is $(\mathrm{Fr}) ^{2\mathbb{Z}}$ and $V_1$ defines a mod $p$ character of $\mathcal{I} _F (\mathrm{Fr}) ^{2\mathbb{Z}}=\mathcal{W} _{E_0}$. Expressing this character as $(\xi \circ \mathbf{a} _{E_0})$ with a suitable character $\xi$ of $E_0^{\times}$, we obtain $\rho \cong \rho _{\xi}$. 
 \end{proof}
 \begin{Rem} \label{Rem:modGalois}
  Let $E_0/F$ be an unramified quadratic extension and $E/F$ a totally ramified quadratic extension. Let us take a uniformizer $\varpi _E \in E^{\times}$ so that $\varpi_E^2 \in F^{\times}$.
  \begin{enumerate}
   \item If $\nu$ is a regular mod $p$ character of $E^{\times}$ (in which case $p\neq 2$), then $\rho _{\nu}$ is isomorphic to $\rho _{\xi}$ for a mod $p$ character $\xi$ of $E_0^{\times}$ such that
         \[ \xi (\zeta _{E_0})^2=\nu (\zeta _F) \text{ and } \xi (\varpi _F)=\nu (-1)^{(q+1)/2} \nu (\varpi _F).\]
           with $\zeta _{E_0}, \zeta _F \text{ and } \varpi_F$ as in \emph{Remark }$\ref{Rem:regchar}.\ref{unramnotation}$.
   \item If $\xi$ is an irregular mod $p$ character of $E_0^{\times}$, any extension of $(\xi \circ \mathbf{a} _{E_0})$ to $\mathcal{W} _F$ is expressed as $(\varphi \circ \mathbf{a} _F)$ with some mod $p$ character $\varphi$ of $F^{\times}$ such that
         \[ \varphi (\zeta _F) =\xi (\zeta _{E_0}) \text{ and } \varphi (\varpi _F)^2=\xi (\varpi _F). \]
           with $\zeta _{E_0}, \zeta _F \text{ and } \varpi _F$ as above.
   \item If $\nu$ is an irregular mod $p$ character of $E^{\times}$, any extension of $(\nu \circ \mathbf{a} _{E})$ to $\mathcal{W} _F$ is expressed as $(\varphi \circ \mathbf{a} _F)$ with some mod $p$ character $\varphi$ of $F^{\times}$. If $p\neq 2$, then $\varphi$ satisfies
         \[ \varphi (\zeta _F)^2=\nu (\zeta _F) \text{ and } \varphi (-\varpi _F)=\nu (\varpi _E). \]
           with $\zeta _F, \varpi _E \text{ and } \varpi _F$ as in \emph{Remark }$\ref{Rem:regchar}.\ref{tameramnotation}$, and if $p=2$, then $\varphi$ is the unique mod $p$ character such that $\varpi ^2=\nu |_{F^{\times}}$.
  \end{enumerate}
 \end{Rem}
 
\subsection{Mod $p$ correspondence} \label{subsection:modpcor}
 By Propositions \ref{Prop:modGalois} and \ref{Prop:modquaternion}, two-dimensional irreducible mod $p$ representations of $\mathcal{W} _F$ and irreducible mod $p$ representations of $D^{\times}$ of dimension greater than one are parametrized by the same set $\mathbb{P} _2 (F, \overline{\mathbb{F}} _p )$ and therefore naturally correspond to each other. However, motivated by the correspondences in characteristic zero (cf. \ref{subsec:corres}), we adjust the correspondence slightly by composing a certain permutation of $\mathbb{P} _2 (F, \overline{\mathbb{F}} _p )$.
 \begin{Def}
  Let $E/F$ be an unramified quadratic extension of $F$.
  
  We define $\delta$ to be the unramified mod $p$ character of $E^{\times}$ $($i.e. $\delta$ is trivial on $U_E)$ sending any uniformizer to $-1$. $($In particular, $\delta$ is trivial if $p=2.)$
 \end{Def}
 It can immediately be seen that the association $\xi \mapsto \delta \otimes \xi$ induces a bijection $\mathbb{P} _2 (F, \overline{\mathbb{F}} _p ) \rightarrow \mathbb{P} _2 (F, \overline{\mathbb{F}} _p )$.
 \begin{Def}
 We call the bijection between $\mathcal{G} _2^0(F, \overline{\mathbb{F}} _p)$ and $\mathcal{A} _1^0(D, \overline{\mathbb{F}} _p )$ induced by the association $\rho _{\xi} \mapsto \pi _{\delta \otimes \xi}$ \emph{the mod $p$ correspondence}.
 \end{Def}
 \begin{Rem}
 As noted in \ref{section:intro}, this correspondence has already been established by Vign\'eras $($cf. \emph{\cite{Vi2}}$)$.
 \end{Rem}

\section{Reduction modulo $p$ of representations of $D^{\times}$ and of $\mathcal{W} _F$} \label{section:reduction}
\subsection{Review of irreducible admissible representations of $D^\times$}

Recall that we fixed an isomorphism $\overline{\mathbb{Q}} _p \cong \mathbb{C}$ in \ref{no,te}. As the topology of the coefficient field is irrelevant to the notion of admissible representations, each irreducible admissible representation of $D^\times$ over $\mathbb{C}$ corresponds to that over $\overline{\mathbb{Q}} _p$ through this isomorphism.

Here we first give a review of the classification of irreducible admissible representations of $D^\times$ over $\mathbb{C}$, restricting ourselves to the \emph{tame} case in the sense below. (Note that an irreducible admissible representation of $D^\times$ is automatically tamely ramified unless $p=2$ or it is a character.) These representations are parametrized by admissible pairs with values in $\mathbb{C}$. All the proofs can essentially be found in [BH] except that they give a complete classification of irreducible admissible representations of both $GL_2(F)$ and $D^\times$ while the parametrization of tame (cuspidal) representations are described only in the case of $GL_2(F)$.
Then we briefly recall what we need about the remaining class of irreducible admissible representations. Although our treatment is incomplete in view of the theory of types it suffices for our purposes.

Irreducible admissible representations of $D^\times$ consist of one-dimensional representations, which factor through $\mathrm{Nrd}: D^\times \rightarrow F^\times$, and $n$-dimensional representations with $1<n< \infty$. We denote the set of isomorphism classes of the latter representations by $\mathcal{A}_1^0(D)$. Elements in $\mathcal{A}_1^0(D)$ correspond to irreducible cuspidal representations of $GL_2(F)$ via LJLC and then in turn to irreducible two-dimensional representations of $\mathcal{W} _F$ via LLC.

\begin{Def}
Let $\Pi$ represent an element in $\mathcal{A}_1^0(D)$.

The representation $\Pi$ is said to be \emph{unramified} if there exits a non-trivial unramified character $\Phi$ of $F^\times$ such that $(\Phi \circ \mathrm{Nrd}) \otimes \Pi \cong \Pi$, and \emph{totally ramified} otherwise. Also, it is said to be \emph{tamely ramified} if $p \neq 2$ or it is unramified.
\end{Def}
We denote the set of isomorphism classes of unramified representations of $D^\times$ by $\mathcal{A}_1^{\emph{nr}}(D)$.

We are now going to construct a map $\mathbb{P}_2(F) \rightarrow \mathcal{A}_1^0(D)$.

Let us begin with an admissible pair $(E/F, \chi)$ \emph{of level zero}. Then $E/F$ is unramified by the definition of admissible pairs. We take an $F$-embedding $E \rightarrow D$. We extend $\chi$ to a character $\Lambda$ of $J=E^{\times} U_D^1(=F^{\times} U_D)$ by setting $\Lambda |_{U_D^1} =1$. Then we set
\[\Pi_{\chi}=\Ind_{J}^{D^{\times}}\chi.\]
It is an irreducible admissible representation of $D^{\times}$.

Now let $(E/F, \chi)$ be a \emph{minimal} pair of positive level $m>0$. We fix a character $\psi$ of $F$ of level one. Let us set $\psi_D = \psi \circ \mathrm{Trd}$, $\psi_E = \psi \circ \tr_{E/F}$ and $n=2m/{e(E|F)}$. The characters $\psi_D$ and $\psi_E$ are both of level one. We take an $F$-embedding $E \rightarrow D$ and identify $E$ with an $F$-subalgebra of $D$.
There exists an $\alpha \in \mathfrak{p}_{E}^{-m}$ such that $\chi (1+x)= \psi_{E} (\alpha x)$ for all $x \in \mathfrak{p}_E^{[m/2]+1}$. (Here, $[l]$ denotes the greatest integer not greater than $l$.) Then $n=-v_D(\alpha)$. Similarly, we define a character $\psi_{\alpha}$ of $U_D^{[n/2]+1}$ trivial on $U_D^{n+1}$ by
\[\psi_{\alpha}(1+x)= \psi _{D} (\alpha x)\]
for all $x \in {\mathfrak{q}}^{[n/2]+1}$. We would like to define an element in the set $C(\psi_{\alpha})$ of isomorphism classes of irreducible representations $\Lambda$ of the group $J_{\alpha} =E^{\times}U_D^{[{(n+1)}/2]}$ such that ${\Lambda}|_{U_D^{[n/2]+1}}$ is a multiple of $\psi_{\alpha}$.

First we treat the case where $n$ is odd. Then $E/F$ is totally ramified and $J_{\alpha} =E^{\times}U_D^{[n/2]+1}$. In this case we define a character $\Lambda$ of $J_{\alpha}$ by
\[{\Lambda}|_{U_D^{[n/2]+1}} =\psi_{\alpha}, \ {\Lambda}|_{E^{\times}}= \chi .\]
Then $\Lambda$ defines a class in $C(\psi_{\alpha})$.

Next we move on to the case where $n\equiv 0\pmod 4$. This case is almost the same as the previous one, except that the extension $E/F$ is unramified. We have $J_{\alpha}=E^{\times}U_D^{[n/2]+1}$ and define a character $\Lambda$ of $J_{\alpha}$ by
\[{\Lambda}|_{U_D^{[n/2]+1}} =\psi_{\alpha}, \ {\Lambda}|_{E^{\times}}= \chi .\]
Again, $\Lambda$ defines a class in $C(\psi_{\alpha})$.

Finally, if $n\equiv 2\pmod 4$, $E^{\times}U_D^{[n/2]+1}$ has index $q^2$ in $J_{\alpha}$. Let us set $H_{\alpha}^1 =U_E^1 U_D^{[n/2]+1} \subsetneq J_{\alpha}^1 =U_E^1 U_D^{[{(n+1)}/2]}$. Defining
\[\theta (ux)=\chi (u)\psi_{\alpha} (x)\]
for $u\in U_E^1$ and $x\in U_D^{[n/2]+1}$, we obtain a character $\theta$ of $H_{\alpha}^1$.
\begin{Prop} \label{Prop:Heisenberg}
 \begin{enumerate}
  \item Under the conditions above there exists an irreducible representation $\eta _{\theta}$ of $J_{\alpha}^1$ containing $\theta$. It is unique up to isomorphism and $q$-dimensional.
  \item Furthermore, there exists an irreducible representation $\Lambda$ of $J_{\alpha}$ such that
  \begin{enumerate}
   \item $\Lambda |_{J_{\alpha}^1} \cong \eta _{\theta}$
   \item $\Lambda |_{F^{\times}} \cong (\chi |_{F^{\times}})^{\oplus q}$
   \item $\tr \Lambda (\zeta)=-\chi (\zeta) \ \ \ (\zeta \in \mu_E \setminus \mu_F)$.
  \end{enumerate}
  It is unique up to isomorphism.
 \end{enumerate}
\end{Prop}
The obtained representation $\Lambda$ defines a class in $C(\psi _{\alpha})$.

Thus we now have an irreducible representation $\Lambda$ of $J_{\alpha}$ in all cases. We set
\[\Pi_{\chi}=\Ind_{J_{\alpha}}^{D^{\times}} \Lambda . \]

Now let $(E/F, \chi)$ be a (not necessarily minimal) admissible pair. There exist a minimal pair $(E/F, \chi^{\prime})$ and a character $\Phi$ of $F^{\times}$ such that $\chi =(\Phi \circ \mathrm{N}_{E/F}) \otimes \chi ^{\prime}$. We define
\[\Pi _{\chi}= (\Phi \circ \mathrm{Nrd}) \otimes \Pi _{\chi^{\prime}}.\]

In the course of the definition of $\Pi _{\chi}$, we made many choices. Nonetheless, we have
\begin{Prop}
 The isomorphism class of $\Pi _{\chi}$ only depends on the isomorphism class of an admissible pair $(E/F, \chi)$. Moreover, the representations $\Pi _{\chi}$ constructed above are irreducible and admissible. In other words, the construction above indeed induces a well-defined map $\mathbb{P}_2(F) \rightarrow \mathcal{A}_1^0(D)$.
\end{Prop}

\begin{Thm} \label{Thm:quaternion}
 Sending the isomorphism class $[(E/F, \chi)]$ of an admissible pair $(E/F, \chi)$ to the isomorphism class $[\Pi _{\chi} ]$ of an irreducible admissible representation $\Pi_{\chi}$ we obtain a bijection $\Pi$ between the set $\mathbb{P}_2(F)$ of isomorphism classes of admissible pairs and the
set of isomorphism classes of tamely ramified representations of $D^{\times}$:
\[\mathbb{P}_2(F) \simeq \mathcal{A}_1^0(D), \ \ [(E/F, \chi)] \mapsto [\Pi _{\chi} ] \qquad \text{if} \ p \neq 2\]
\[\mathbb{P}_2(F) \simeq \mathcal{A}_1^{\emph{nr}}(D), \ \ [(E/F, \chi)] \mapsto [\Pi _{\chi} ] \qquad \text{if} \ p=2.\]
If $(E/F, \chi)$ is an admissible pair, we have
 \begin{enumerate}
  \item the level of $\Pi_{\chi}$ is $n= {2m}/e(E|F)$,
  \item the central character of $\Pi_{\chi}$ is $\chi |_{F^{\times}}$,
  \item $\Pi_{(\Phi \circ \mathrm{N}_{E/F}) \otimes \chi} \cong (\Phi \circ \mathrm{Nrd}) \otimes \Pi_{\chi}$ \ \  if $\Phi$ is a character of $F^{\times}$,
  \item and $\Pi _{\chi}$ is unramified if and only if $E/F$ is unramified.
 \end{enumerate}
\end{Thm}

We only state the corresponding theorem for not tamely ramified representations in a very crude form. The proof of the theorem in a more refined form is found again in \cite{BH}.
\begin{Thm}\label{Thm:quaternion wild}
 Assume that $p=2$. Let $\Pi$ represent an element in $\mathcal{A}_1^0(D) \setminus \mathcal{A}_1^{\emph{nr}}(D)$. Let $n \in \mathbb{Z}$ be the minimal integer among the levels of representations obtained from $\Pi$ by a character twist.
 Then $n$ is odd and there exists a totally ramified quadratic extension $E/F$ embedded in $D$ and a character $\chi$ of $E^{\times}U_D^{(n+1)/2}$ such that $\Pi \cong \Ind _{E^{\times}U_D^{(n+1)/2}} ^{D^{\times}} \chi$. 
\end{Thm}

\subsection{Reduction of irreducible representations of $D^\times$} \label{subsection:quaternionred}

First let us clarify what we mean by reduction. If $\Pi$ is a representation of finite length, $(\Pi) ^{\text{ss}}$ denotes the semisimplification (i.e. the direct sum of composition factors with the appropriate multiplicities) of $\Pi$. Every representation $(\Pi, V)$ of a locally profinite group $G$ over $\overline{\mathbb{Q}} _p$ considered in this paper is finite-dimensional and can be realized over a finite extension $K$ of $\mathbb{Q} _p$. Let $A$ be the ring of integers in $K$ and $\mathfrak{m}$ the maximal ideal of $A$. The representation $(\Pi, V)$ is said to be $p$-integral if there exists an $A$-lattice $\Lambda \subset V$ stable under $G$. If $\Pi$ is $p$-integral, the semisimplification of the representation $\Lambda /\mathfrak{m} \Lambda$ over $A/\mathfrak{m}$ is independent up to isomorphism of the choice of a lattice $\Lambda$ (cf. \cite[Chap.15.2]{Se1}, \cite[Chap1.9]{Vi1}). We consider the semisimplification as a representation over $\overline{\mathbb{F}} _p$, call it, by abuse of language,  \emph{the reduction modulo $p$} (or even more simply, \emph{the reduction}) of $\Pi$ and denote it by $\overline{\Pi}^{\text{ss}}$.\footnote{Any irreducible representation of $D^{\times}$ factors through a finite quotient group after a suitable character twist. Therefore we can compute the reduction of irreducible representations by means of Brauer character. For a similar computation with Brauer character, see \cite[Appendice]{BD}. } If $\chi$ is a $p$-integral character, then we often simply denote its reduction by $\overline{\chi}$.

Next we note the following two facts which are simple, yet very useful for the computation of the reduction of representations of $D^{\times}$.
\begin{itemize}
 \item If $\sigma$ is a $p$-integral representation of an open subgroup $U$ of $D^{\times}$, $\Ind_U^{D^{\times}}\sigma$ is also $p$-integral and the reduction of the representation $\Ind_U^{D^{\times}}\sigma$ is isomorphic to the semisimplification of the (smooth) induction $\Ind_U^{D^{\times}}\overline{\sigma}^{\text{ss}}$ of the reduction $\overline{\sigma}^{\text{ss}}$ of $\sigma$.
 \item If $\sigma$ is a $p$-integral representation of $E^{\times}U_D^1$, where $E$ is unramified over $F$ (resp. where $E$ is totally ramified over $F$), and $E$ is considered as an $F$-subalgebra of $D$ via an $F$-embedding, the reduction of $\sigma$ can be computed by inflating the reduction of the restricted representation 
\[\Res_{\mu _{E} \times \varpi _F^{\mathbb{Z}}}^{E^{\times}U_D^1} \sigma \ \ \ \left(\text{resp.} \ \Res_{\mu _{F}\times \varpi_{E} ^{\mathbb{Z}}}^{E^{\times}U_D^1} \sigma \right)\]
to $E^{\times}U_D^1$, where $\varpi_F$ (resp. $\varpi_E$) denotes a uniformizer of $F$ (resp. $E$).

(Indeed, since $E^{\times}U_D^1$ is a semidirect product of the normal pro-$p$-subgroup $U_D^1$ and a locally profinite subgroup of the form above, every irreducible mod $p$ representation is inflated from an irreducible representation of the latter group. Thus a composition series of a mod $p$ representation of $E^{\times}U_D^1$ serves as a composition series of the restricted representation.)
\end{itemize}
The computation of the reduction in the tame case relies on the following key lemmas, which are variants of \cite[16.2 Lemma]{BH}.

Let $(E/F, \chi)$ be a minimal pair of positive level $m>0$. Let us set $n={2m}/e(E|F)$, take $\alpha \in \mathfrak{p}_E^{-m}$ and choose an $F$-embedding $E\rightarrow D$ exactly as in the previous subsection.
\begin{lem}
\label{lem:unram}
 Suppose that $n$ is even $($in which case $E/F$ is necessarily unramified$)$. Let us fix a generator $\zeta _{E} \in \mu _{E}$ and a uniformizer $\varpi _F$ of $F$. Then,
 \begin{equation*}
  \begin{split}
   \left( \mu _E \times \varpi _F^{\mathbb{Z}} \right) \Big\backslash E^{\times}U_D^1\Big/ J_{\alpha} \ \ \ &\left( \cong \ \underset{\mu _E}{\sim} \Big\backslash U_D^1\Big/U_E^1U_D^{[{(n+1)}/2]} \right) \\
   & = \{ J_{\alpha} \} \amalg \left \{ \coprod_{0\leq i\leq q} \zeta _E^ix J_{\alpha} \Biggm| x\in U_D^1 \setminus U_E^1U_D^{[{(n+1)}/2]} \right \} ,
  \end{split}
 \end{equation*}
 where $\underset{\mu _E}{\sim}$ stands for the equivalence relation induced by the left action of $\mu _E$ on $U_D^1\Big/U_E^1U_D^{[(n+1)/2]}$ defined by
 \[\zeta \cdot xU_E^1U_D^{[{(n+1)}/2]}= (\zeta x\zeta ^{-1})U_E^1U_D^{[{(n+1)}/2]} \]
 for $\zeta \in \mu_E$ and $x\in U_D^1$.

 $\Big($i.e. if $\zeta \in \mu _E \setminus \mu _F \ \text{and} \ u\in U_D^1$,
 \[\zeta u\zeta ^{-1} \equiv u \pmod{U_E^1U_D^{[{(n+1)}/2]}} \ \text{holds if and only if} \ u\in U_E^1U_D^{[{(n+1)}/2]}. \Big) \]
 In particular, we have
 \[\# \left( \left(\mu _E \times \varpi _F^{\mathbb{Z}} \right) \Big\backslash E^{\times}U_D^1\Big/ J_{\alpha} \right) =\frac{q^{2[n/4]}-1}{q+1} +1 \text{.}\]
\end{lem}

We only consider a particularly simple situation in the corresponding lemma for ramified extensions.
\begin{lem}
\label{lem:ram}
 Suppose that $n$ is odd $($in which case $E/F$ is necessarily totally ramified and \emph{$p\neq 2$}$)$. Let us fix a uniformizer $\varpi _E$ of $E^{\times}$ such that $\varpi _E^2 \in F^{\times}$. Then,
 \begin{equation*}
  \begin{split}
   \left(\mu _F \times \varpi _E^{\mathbb{Z}}\right) \Big\backslash E^{\times}U_D^1\Big/ J_{\alpha} \ \ \ &\left(\cong \ \underset{\varpi _E^{\mathbb{Z}}}{\sim} \Big\backslash U_D^1\Big/U_E^1U_D^{[{(n+1)}/2]}\right) \\
   &= \{ J_{\alpha} \} \amalg \left \{ \coprod_{0\leq i\leq 1} \varpi _E^ix J_{\alpha} \Biggm| x\in U_D^1 \setminus U_E^1U_D^{[{(n+1)}/2]}\right \},
  \end{split}
 \end{equation*}
 where $\underset{\varpi _E^{\mathbb{Z}}}{\sim}$ stands for the equivalence relation induced by the left action of $\varpi _E^{\mathbb{Z}}$ on $U_D^1\Big/U_E^1U_D^{[(n+1)]}$ defined by
 \[\varpi _E^i \cdot xU_E^1U_D^{[{(n+1)}/2]}= (\varpi _E^i x\varpi _E^{-i})U_E^1U_D^{[{(n+1)}/2]} \]
 for $i \in \mathbb{Z}$ and $x\in U_D^1$.
 
 $\Big($i.e. if $u\in U_D^1$,
 \[\varpi _E u\varpi _E ^{-1} \equiv u \pmod{U_E^1U_D^{[{(n+1)}/2]}} \ \text{holds if and only if} \ u\in U_E^1U_D^{[{(n+1)}/2]}.\Big) \]
 In particular, we have
 \[\# \left(\left(\mu _F \times \varpi _E^{\mathbb{Z}}\right) \Big\backslash E^{\times}U_D^1\Big/ J_{\alpha}\right) =\frac{q^{{(n-1)}/2}-1}{2} +1 \text{.}\]
\end{lem}

Admitting these lemmas for a moment, we compute the reduction first.

Let $(E/F, \chi)$ be an admissible pair and let us choose an $F$-embedding $E\rightarrow D$. Let us take a decomposition $\chi =(\Phi \circ \mathrm{N}_{E/F})\otimes \chi ^{\prime}$ with $\Phi$ a character of $F^{\times}$ and $(E/F, \chi)$ a minimal pair, denote by $m$ the level of $\chi ^{\prime}$ and set $n= {2m}/e(E/F)$. 
\begin{Thm} \label{Thm:quaternionred}
 Let the notation be as above.

 The irreducible admissible representation $\Pi _{\chi}$ is $p$-integral if and only if $\chi$ is $p$-integral.

 Moreover, assume that $\chi$ is a $p$-integral character.
 \begin{enumerate}

  \item \label{Thm:unram}
  Assume that $n$ is even $($in which case $E/F$ is necessarily unramified$)$. Then we have
   \begin{equation*}
    \overline{\Pi } _{\chi} ^{\emph{ss}} \cong \begin{cases}
                                \left(\pi _{\overline{\chi}}^{\oplus \frac{q^{n/2}+q}{q+1}} \oplus \left(\bigoplus _{\tau \in X} \pi _{\overline{\chi} \otimes \tau}^{\oplus \frac{q^{n/2}-1}{q+1}}\right)\right)^{\emph{ss}} \qquad & \text{if} \ n \equiv 0\pmod {4} \text{,} \\
                                \left(\pi _{\overline{\chi}}^{\oplus \frac{q^{n/2}-q}{q+1}} \oplus \left(\bigoplus _{\tau \in X} \pi _{\overline{\chi} \otimes \tau}^{\oplus \frac{q^{n/2}+1}{q+1}}\right)\right)^{\emph{ss}} \qquad & \text{if} \ n \equiv 2\pmod {4} \text{,} 
                                                                   \end{cases}
   \end{equation*}
   where $X$ is the set of non-trivial mod $p$ characters of $E^{\times}$ which are trivial on $F^{\times}U_E^1$ $($consisting of $q$ elements$)$.

   Among mod $p$ representations of the form $\pi _{\overline{\chi}}$ or $\pi _{\overline{\chi} \otimes \tau}$ for some $\tau \in X$, there exists a reducible representation if and only if $\overline{\chi} (-1)= 1.$ 


   

  \item \label{Thm:ram}
  Assume that $n$ is odd $($in which case $E/F$ is necessarily totally ramified and \emph{$p\neq 2$}$)$. Let us denote by $\overline{\chi} ^{\ddagger}$ the extension of $\overline{\chi}$ to $E^{\times}U_D^1$ such that $\overline{\chi} ^{\ddagger}|_{U_D^1}=1$ as in Proposition \ref{Prop:quaternionmodpram}. Then we have
   \begin{equation*}
    \begin{split}
     \overline{\Pi } _{\chi} ^{\emph{ss}} &\cong \left(\Ind _{E^{\times}U_D^1}^{D^{\times}} \overline{\chi} ^{\ddagger}\right)^{\oplus \frac{q^{{(n-1)}/2}+1}{2}} \oplus \left(\Ind _{E^{\times}U_D^1}^{D^{\times}} (\overline{\chi} ^{\ddagger} \otimes \iota )\right)^{\oplus \frac{q^{{(n-1)}/2}-1}{2}} \\
                       &\cong \left(\bigoplus _{\pi \in I_1} \pi ^{\oplus q^{({n-1})/2}}\right) \oplus \left(\bigoplus _{\tilde{\chi} \in I_2} \tilde{\chi} ^{\oplus \frac{q^{{(n-1)}/2}+1}{2}}\right) \oplus \left(\bigoplus _{\tilde{\chi} \in I_2} (\tilde{\chi} \otimes \tilde{\iota}) ^{\oplus \frac{q^{{(n-1)}/2}-1}{2}} \right),
    \end{split}
   \end{equation*}
   where $\iota$ $($\emph{resp.} $\tilde{\iota})$ denotes the mod $p$ character of $E^{\times}U_D^1$ $($\emph{resp.} $D^{\times})$ trivial on $U_EU_D^1$ $($\emph{resp.} $U_D)$ sending any uniformizer $\varpi _E$ of $E^{\times}$ to $-1$, $I_1$ is the set of $($isomorphism classes of$)$ two-dimensional irreducible mod $p$ representations of $D^{\times}$ with a central character $\overline{\chi}|_{F^{\times}}$ and $I_2$ is the set of mod $p$ characters of $D^{\times}$ extending $\overline{\chi}^{\ddagger}$.
   
   If $E_0$ denotes an unramified quadratic extension of $F$, we have
   \[I_1= \{ \pi _{\tilde{\chi}} \mid \tilde{\chi} \ \text{is a regular mod} \ p \ \text{character of} \ E_0^{\times} \ \text{extending} \ \overline{\chi} |_{F^{\times}}\} ,\]
   \[I_2= \{ \varphi \circ \mathrm{Nrd} \mid \varphi \ \text{is a mod} \ p \ \text{character of} \ F^{\times} \ \text{such that} \ (\varphi \circ \mathrm{Nrd}) |_{E^{\times}}=\overline{\chi} \} ,\]
   and
   \begin{equation*}
    (\# I_1, \ \# I_2)= \begin{cases}
                         (\frac{q+1}{2}, \ 0) & \text{if} \ \chi (-1)=-1 \\
                         (\frac{q-1}{2}, \ 2) & \text{if} \ \chi (-1)=1 .
                        \end{cases}
   \end{equation*}
 \end{enumerate}
\end{Thm}
\begin{Rem} \label{Rem:quaternionred}
 In most cases $($namely, except for the cases where $n=0$ and where $n=1, 2$ and  $\overline{\chi}$ is irregular$)$ every irreducible admissible mod $p$ representation of $D^{\times}$ with a central character $\overline{\chi} |_{F^{\times}}$ occurs in $\overline{\Pi} _{\chi}^{\emph{ss}}$.
\end{Rem}

\begin{proof}
 We are immediately reduced to show the theorem for \emph{minimal} representations (i.e. ones which are parametrized by minimal pairs).
 
 If $\chi$ is $p$-integral, $\Lambda$ is $p$-integral and therefore $\Pi _{\chi}$ is $p$-integral. Conversely, if $\Pi _{\chi}$ is $p$-integral, the central character $\chi |_{F^{\times}}$ is $p$-integral and so is $\chi$.
 
 In the proof of this theorem and the next one, we denote by $1_G$ the trivial mod $p$ character of a locally profinite group $G$.
 
 First assume that $n$ is even and $\chi$ is $p$-integral.
 
 If $n=0$, the reduction $\overline{\Lambda}^{\text{ss}}$ of the representation $\Lambda$ of $J$ is $\overline{\chi}^{\dagger}$ and the assertion is obvious.
 
 For $n>0$, the reduction $\overline{\Lambda}^{\text{ss}}$ of $\Lambda$ is trivial on the normal pro-$p$-subgroup $U_D^{[(n+1)/2]}$ and it is a direct sum of irreducible mod $p$ representations inflated from representations of an abelian group $J_{\alpha}/U_D^{(n+1)/2}$. The composition factors can be read off by restricting $\Lambda$ to $\mu _E F^{\times}$ (cf. Proposition \ref{Prop:Heisenberg}):
 \begin{equation*}
  \overline{\Lambda}^{\text{ss}} \cong \begin{cases}
                       \overline{\chi} ^{\dagger} |_{J_{\alpha}} & \text{if } n \equiv 0 \pmod 4 \\
                       \bigoplus _{\tau \in X} (\overline{\chi} \otimes \tau) ^{\dagger} |_{J_{\alpha}} & \text{if } n \equiv 2 \pmod 4.
                      \end{cases}
 \end{equation*}
So we have
  \begin{equation*}
  \begin{split}
   \overline{\left(\Ind_{J_{\alpha}}^{E^{\times}U_D^1} \Lambda\right)}^{\text{ss}} &\cong \left(\Ind_{J_{\alpha}}^{E^{\times}U_D^1} \overline{\Lambda}^{\text{ss}}\right)^{\text{ss}} \\
   &\cong \begin{cases}
           \overline{\chi}^{\dagger} \otimes \left(\Ind_{J_{\alpha}}^{E^{\times}U_D^1} 1_{J_{\alpha}}\right)^{\text{ss}} & \text{if } n \equiv 0 \pmod 4 \\
           \left(\bigoplus _{\tau \in X} (\overline{\chi} \otimes \tau)^{\dagger}\right) \otimes \left(\Ind_{J_{\alpha}}^{E^{\times}U_D^1} 1_{J_{\alpha}}\right)^{\text{ss}} &\text{if } n \equiv 2 \pmod 4.
          \end{cases}
  \end{split}
  \end{equation*}
 Thus, we are reduced to show
 \begin{equation*}
   \left(\Ind_{J_{\alpha}}^{E^{\times}U_D^1} 1_{J_{\alpha}}\right)^{\text{ss}} \cong
  \begin{cases}
  1_{E^{\times}U_D^1} \oplus \left(\left(\Ind_{F^{\times}U_D^1}^{E^{\times}U_D^1} 1_{F^{\times}U_D^1}\right)^{\oplus \frac{q^{n/2}-1}{q+1}}\right)^{\text{ss}} & \text{if } n \equiv 0 \pmod 4 \\
  1_{E^{\times}U_D^1} \oplus \left(\left(\Ind_{F^{\times}U_D^1}^{E^{\times}U_D^1} 1_{F^{\times}U_D^1}\right)^{\oplus \frac{q^{(n-2)/2}-1}{q+1}}\right)^{\text{ss}} & \text{if } n \equiv 2 \pmod 4,
  \end{cases}
 \end{equation*}
 i.e.
 \[\left(\Ind_{J_{\alpha}}^{E^{\times}U_D^1} 1_{J_{\alpha}}\right)^{\text{ss}} \cong 1_{E^{\times}U_D^1} \oplus \left(\left(\Ind_{F^{\times}U_D^1}^{E^{\times}U_D^1} 1_{F^{\times}U_D^1}\right)^{\oplus \frac{q^{2[n/2]}-1}{q+1}}\right)^{\text{ss}}.\]
 
 Now, if we take a generator $\zeta _E \in \mu _E$ and a uniformizer $\varpi _F$ of $F$, Mackey's decomposition \cite[7.3]{Se1}, \cite[I.5.5]{Vi1} together with Lemma \ref{lem:unram} gives
 \begin{equation*}
  \begin{split}
   \Res_{\mu _E \times \varpi _F^{\mathbb{Z}}}^{E^{\times}U_D^1} \Ind_{J_{\alpha}}^{E^{\times}U_D^1} 1_{J_{\alpha}}
   &\cong \bigoplus _{(\mu _E \times \varpi _F^{\mathbb{Z}})xJ_{\alpha} \atop \in (\mu _E \times \varpi _F^{\mathbb{Z}}) \big\backslash E^{\times}U_D^1 \big/ J_{\alpha}} \Ind_{J_{\alpha}^{(x)}}^{\mu _E \times \varpi _F^{\mathbb{Z}}} 1_{J_{\alpha}^{(x)}} \\
   &\cong 1_{\mu _E \times \varpi _F^{\mathbb{Z}}} \oplus \Big( \bigoplus _{xU_E^1U_D^{[(n+1)/2]}} \Ind_{J_{\alpha}^{(x)}}^{\mu _E \times \varpi _F^{\mathbb{Z}}} 1_{J_{\alpha}^{(x)}}\Big),
  \end{split}
 \end{equation*}
 where $xU_E^1U_D^{[(n+1)/2]}$ runs through $\big( U_D^1\big/U_E^1U_D^{[(n+1)/2]}\big) \setminus \big\{ U_E^1U_D^{[(n+1)/2]} \big\}$ and $J_{\alpha}^{(x)}$ denotes $(\mu _E \times \varpi _F^{\mathbb{Z}}) \cap xJ_{\alpha}x^{-1}$ for $x \in E^{\times}U_D^1$.

 Here we have $J_{\alpha}^{(x)}=\mu _F \times \varpi _F^{\mathbb{Z}}$ for any $x \in U_D^1\setminus U_E^1U_D^{[(n+1)/2]}$ and $n>2$. Indeed, the index of $J_{\alpha}^{(x)}$ in $\mu _E \times \varpi _F^{\mathbb{Z}}$ is equal to the length of the $\mu _E$-orbit containing $xU_E^1U_D^{[(n+1)/2]}$, namely, to $q+1$.
 The subset $J_{\alpha}^{(x)}$ surely contains $\mu _F \times \varpi _F^{\mathbb{Z}}$ and the desired equality follows.\footnote{Alternatively one can proceed as follows. One inclusion of the equality being trivial, we are to show
 \[ \text{if } x \in U_D^1 \text{ satisfies } \ xyx^{-1} \in \mu _E \setminus \mu _F \text{ with } y \in J_{\alpha}, \text{ then } x \in U_E^1U_D^{[(n+1)/2]} .\]
 Let us put $\zeta =xyx^{-1} \in \mu _E \setminus \mu _F.$ Comparing the images under $v_D$ and the canonical projection $U_D \rightarrow U_D/U_D^1$, we see that $yU_D^1=\zeta U_D^1$. Hence $\zeta x\zeta ^{-1} =xy\zeta ^{-1} \equiv x \pmod{U_D^1}$ and the claim now follows from Lemma \ref{lem:unram}.}
 
 Hence, inflating the representation from $\mu _E \times \varpi _F^{\mathbb{Z}}$ to $E^{\times}U_D^1$, we obtain
 \[ \left(\Ind_{J_{\alpha}}^{E^{\times}U_D^1} 1_{J_{\alpha}}\right)^{\text{ss}} \cong 1_{E^{\times}U_D^1} \oplus \left(\left(\Ind_{F^{\times}U_D^1}^{E^{\times}U_D^1} 1_{F^{\times}U_D^1}\right)^{\oplus \frac{q^{[n/4]}-1}{q+1}} \right)^{\text{ss}} \]
 as required. Thus we have the required decomposition of $\overline{\Pi} _{\chi}$. The assertion about the existence of reducible representations of the form $\pi _{\overline{\chi}}$ or $\pi _{\overline{\chi} \otimes \tau}$ can be verified in much the same way as the proof of similar assertions in Proposition \ref{Prop:quaternionmodpram}. This completes the proof of \ref{Thm:unram}.
 
 Next assume that $n$ is odd and $\chi$ is $p$-integral. The strategy for the computation is exactly the same as that in the case of unramified representations.
 
 The reduction $\overline{\Lambda}^{\text{ss}}$ of $\Lambda$ is isomorphic to $\overline{\chi}^{\ddagger}|_{J_{\alpha}}$ and the reduction $\overline{(\Ind_{J_{\alpha}}^{E^{\times}U_D^1} \Lambda)}^{\text{ss}}$ of $\Ind_{J_{\alpha}}^{E^{\times}U_D^1} \Lambda$ is given by
 \begin{equation*}
  \begin{split}
   \overline{\left(\Ind_{J_{\alpha}}^{E^{\times}U_D^1} \Lambda\right)}^{\text{ss}} &\cong \left(\Ind_{J_{\alpha}}^{E^{\times}U_D^1} \overline{\Lambda}^{\text{ss}}\right)^{\text{ss}} \\
                                                       &\cong \overline{\chi}^{\ddagger} \otimes \left(\Ind_{J_{\alpha}}^{E^{\times}U_D^1} 1_{J_{\alpha}}\right)^{\text{ss}},
  \end{split}
 \end{equation*}
 
 Then Mackey's decomposition \cite[7.3]{Se1}, \cite[I.5.5]{Vi1} and Lemma \ref{lem:ram} yield in turn
 \[ \left(\Ind_{J_{\alpha}}^{E^{\times}U_D^1} 1_{J_{\alpha}}\right)^{\text{ss}}  
 \cong 1_{E^{\times}U_D^1} \oplus \left(\Ind_{F^{\times}U_D^1}^{E^{\times}U_D^1} 1_{F^{\times}U_D^1}\right)^{\oplus \frac{q^{(n-1)/2}-1}{2}} .\]
 The remaining assertions follow from Proposition \ref{Prop:quaternionmodpram}.
\end{proof}

The theorem for the rest of representations is simpler.
\begin{Thm} \label{Thm:quaternionred wild}
 Assume that $p=2$. Let $\Pi=\Ind_{E^{\times}U_D^{(n+1)/2}} ^{D^{\times}} \chi$ represent an element in $\mathcal{A}_1^0(D) \setminus \mathcal{A}_1^{\emph{nr}}(D)$ $( \emph{cf. Theorem } \ref{Thm:quaternion wild} )$.

 The representation $\Pi$ is $p$-integral if and only if the central character $\Phi$ of $\Pi$ is $(\text{or equivalently, } \chi \text{ is})$ $p$-integral.

 Moreover, assume that $\Pi$ is a $p$-integral representation. Let us denote by $\overline{\chi} ^{\ddagger}$ the extension of $\overline{\chi}$ to $E^{\times}U_D^1$ such that $\overline{\chi} ^{\ddagger}|_{U_D^1}=1$ as in Proposition \ref{Prop:quaternionmodpram}.

 Then we have
   \begin{equation*}
    \begin{split}
     \overline{\Pi } ^{\emph{ss}} &\cong \left(\Ind _{E^{\times}U_D^1}^{D^{\times}} \overline{\chi} ^{\ddagger}\right)^{\oplus q^{(n-1)/2}} \\
                                         &\cong \left(\bigoplus _{\pi \in I_1} \pi ^{\oplus q^{({n-1})/2}}\right) \oplus \tilde{\chi} ^{\oplus q^{(n-1)/2}},
    \end{split}
   \end{equation*}
   where $I_1$ is the set of $($isomorphism classes of$)$ two-dimensional irreducible mod $p$ representations of $D^{\times}$ with a central character $\overline{\Phi} (=\overline{\chi}|_{F^{\times}})$ and $\tilde{\chi}$ is the mod $p$ character of $D^{\times}$ extending $\overline{\Phi}$ $( \text{or equivalently, extending } \overline{\chi}^{\ddagger})$.
\end{Thm}

\begin{Rem} \label{Rem:quaternionred}
 Under the conditions of the theorem every irreducible admissible mod $p$ representation of $D^{\times}$ with a central character $\overline{\Phi}$ occurs in $\overline{\Pi}^{\emph{ss}}$.
\end{Rem}

\begin{proof}
 The proof is parallel to that for the preceding theorem in odd $n$ case.
 
 Assume that $\Pi=\Ind_{E^{\times}U_D^{(n+1)/2}}^{D^{\times}} \chi$ is $p$-integral. Then the reduction $\overline{\left(\Ind_{E^{\times}U_D^{(n+1)/2}}^{E^{\times}U_D^1} \chi \right)}^{\text{ss}}$ is isomorphic to $\overline{\chi}^{\ddagger} \otimes \left(\Ind_{E^{\times}U_D^{(n+1)/2}}^{E^{\times}U_D^1} 1_{E^{\times}U_D^{(n+1)/2}}\right)^{\text{ss}}$. Since $\left(\Ind_{E^{\times}U_D^{(n+1)/2}}^{E^{\times}U_D^1} 1_{E^{\times}U_D^{(n+1)/2}}\right)^{\text{ss}}$ is clearly a direct sum of mod $p$ characters trivial on $F^{\times}$ and an induced representation $\Ind_{E^{\times}U_D^1} ^{D^{\times}} \nu ^{\ddagger}$ only depends on $\nu|_{F^{\times}}$ if $p=2$ by Proposition \ref{Prop:quaternionmodpram}, the theorem follows easily.
\end{proof}

We now prove the Lemmas \ref{lem:unram}, \ref{lem:ram}.

\begin{proof}[proof of Lemma \ref{lem:unram}]
 Let us set $n=2m$. We are going to prove the ``only if'' part of the parenthesized restatement, namely,
 \[ \text{if } u\in U_D^1 \text{ satisfies } \zeta u\zeta ^{-1} \equiv u \pmod{U_E^1U_D^{m}} \text{ with } \zeta \in \mu _E \setminus \mu _F, \text{ then } u\in U_E^1U_D^{m}, \]
 by induction on $m$.
 
 As $U_E^1U_D^1=U_D^1$, the claim for $m=1$ is trivial.
 
 Suppose $m\geqq 2$ and that the claim holds for $m-1$. Then it follows that $u\in U_E^1U_D^{m-1}$ and therefore we may assume $u\in U_D^{m-1}$. Noting that $\mathrm{min}\{h\in \mathbb{Z} \mid 2h\geqq m-1\}=[n/4]$, we see
 \begin{equation*}
  \begin{split}
   U_D^{m-1} \cap (U_E^1U_D^m ) &=U_E^{[n/4]}U_D^m \\
                                   &=\begin{cases}
                                    U_E^{m/2}U_D^m \qquad     &\text{ if } n\equiv 0 \pmod4 \\
                                    U_E^{(m-1)/2}U_D^m \qquad &\text{ if } n\equiv 2 \pmod4
                                     \end{cases} \\
                                   &=\begin{cases}
                                      U_D^m \qquad     &\text{ if } n\equiv 0 \pmod4 \\
                                      U_D^{m-1} \qquad &\text{ if } n\equiv 2 \pmod4.
                                     \end{cases}
  \end{split}
 \end{equation*}
 Hence, if $m$ is odd (i.e. $n \equiv 2\pmod4$), we have $u\in U_D^{m-1} \subseteq U_E^1U_D^{m}$ and there is nothing to prove.

 Suppose, on the other hand, that $m$ is even (i.e. $n \equiv 0\pmod4$). Then we are to show that
 \[ \text{ if } u\in U_D^{m-1} \text{ satisfies } \zeta u\zeta ^{-1} \equiv u \pmod{U_D^{m}} \text{ with } \zeta \in \mu _E \setminus \mu _F, \text{ then } u\in U_D^m. \]
 Let us take an element $\varpi _D$ of $D^{\times}$ such that $v_D(\varpi _D)=1$. Setting $u=1+\varpi _D^{m-1} x$ with $x \in \mathfrak{o}_D$ and expressing the claim in additive terms, we obtain the following restatement:
 \[ \text{if } x\in \mathfrak{o}_D \text{ satisfies } \zeta \varpi _D^{m-1} x\zeta ^{-1} \equiv \varpi _D^{m-1} x \pmod{\mathfrak{q}^m} \text{ with } \zeta \in \mu _E \setminus \mu _F, \text{ then } x\in \mathfrak{q} .\]
 The congruence is equivalent to $\varpi _D^{1-m} \zeta \varpi _D^{m-1} x\zeta ^{-1} \equiv x\pmod{\mathfrak{q}}$. If we denote by $\overline{x}$ and $\overline{\zeta}$ the image of $x$ and of $\zeta$ in $k_D$ respectively, this means $(\overline{\zeta} ^{q-1} -1)\overline{x} =0$. Since $\zeta \notin \mu _F$, we surely have $x\in \mathfrak{q}$.
 
 Observing $\# (U_D^1/U_E^1U_D^{[(n+1)/2]})=q^{2[n/4]}$, we verify the assertion about the number of double cosets easily.
\end{proof}
The other lemma is proved in a very similar manner.
\begin{proof}[proof of Lemma \ref{lem:ram}]
 Let us set $n=2m-1$. We are going to show the claim:
 \[ \text{if } u\in U_D^1 \text{ satisfies } \varpi _E u\varpi _E^{-1} \equiv u\pmod{U_E^1U_D^m}, \text{ then } u\in U_E^1U_D^m, \]
 by induction on $m$.

 The claim for $m=1$ is trivial. Suppose $m\geqq 2$ and that the claim holds for $m-1$. The induction hypothesis shows that $u\in U_E^1U_D^{m-1}$ and we may therefore assume that $u\in U_D^{m-1}$.
 
 We have $U_D^{m-1}\cap (U_E^1U_D^m)=U_E^{m-1}U_D^m$ and hence the claim is equivalent to
 \[ \text{if } u\in U_D^{m-1} \text{ satisfies } \varpi _E u\varpi _E^{-1} \equiv u\pmod{U_E^{m-1}U_D^m}, \text{ then } u\in U_E^{m-1}U_D^m. \]
 If we set $u=1+\varpi _E^{m-1} x$ with $x\in \mathfrak{o}_D$ and express the claim in additive terms, it reduces to
 \[ \text{if } x\in \mathfrak{o}_D \text{ satisfies } \varpi _E^m x\varpi _E^{-1} \equiv \varpi _E^{m-1} x\pmod{\mathfrak{p} _E^{m-1} +\mathfrak{q}^m}, \text{ then } x\in \mathfrak{o} _E +\mathfrak{q} .\]
 Then the congruence is equivalent to $\varpi _E x\varpi _E^{-1} \equiv x\pmod{\mathfrak{o} _E +\mathfrak{q}}$. If we denote by $\overline{x}$ the image of $x$ in $k_D$, this means
 \[ \tr _{k_D/{k_E}} (\overline{x})-2\overline{x} =\overline{x} ^q-\overline{x} \in k_E, \]
 from which it certainly follows that $x \in \mathfrak{o}_E +\mathfrak{q}$, since $p\neq 2$.
 
 In view of the fact that $\# (U_D^1/U_E^1U_D^{[(n+1)/2]})=q^{(n-1)/2}$, the assertion about the number of double cosets is not difficult to show.
\end{proof}

\subsection{Reduction of representations of $\mathcal{W} _F$}
In order to compare the reduction of an irreducible admissible representation $\Pi$ of $D^{\times}$ and the reduction of the two-dimensional representation $R$ of $\mathcal{W} _F$ which corresponds to $\Pi$ under LJLC and LLC, we compute the reduction of two-dimensional representations of $\mathcal{W} _F$ in this subsection. We first take up those representations corresponding to tamely ramified representations of $D^{\times}$ under these correspondences and then the other representations in the case where $p=2$. Let us denote the set of isomorphism classes of two-dimensional irreducible smooth representations of $\mathcal{W} _F$ over $\mathbb{C}\cong \overline{\mathbb{Q}}_p$ by $\mathcal{G} _2^0(F)$.

 \begin{Def}
  We define a subset $\mathcal{G} _2^{\emph{nr}}(F)$ of $\mathcal{G} _2^0(F)$ to be the set of isomorphism classes of two-dimensional irreducible representations $R$ of $\mathcal{W} _F$ such that there exists a non-trivial unramified character $\Phi$ of $F^{\times}$ with $(\Phi \circ \mathbf{a} _F) \otimes R \cong R$.
 \end{Def}
 \begin{Thm} \label{Thm:Galois}
  \begin{enumerate}
%
   \item Let $(E/F, \chi)$ be an admissible pair. Let us set $R_{\chi} =\Ind _{\mathcal{W} _E}^{\mathcal{W} _F} (\chi \circ \mathbf{a} _E)$.
   
         Then the isomorphism class of $R_{\chi}$ only depends on that of $(E/F, \chi)$ and the association $(E/F, \chi)$ $\mapsto$ $R_{\chi}$ induces a bijection
         \[\mathbb{P}_2(F) \simeq \mathcal{G}_2^0(F), \qquad \text{if} \ p \neq 2\]
         \[\mathbb{P}_2(F) \simeq \mathcal{G}_2^{\emph{nr}}(F), \qquad \text{if} \ p=2.\]
         If $(E/F, \chi)$ is an admissible pair, we have
         \begin{enumerate}
          \item $\mathrm{det} R_{\chi}= (\chi |_{F^{\times}} \otimes \varkappa _{E/F}) \circ \mathbf{a} _{F}$,
          
                where $\varkappa _{E/F}$ is the character of $F^{\times}$ associated to the quadratic extension $E/F$ $($i.e. the non-trivial character of $F^{\times}$ trivial on $\mathrm{N} _{E/F} (E^{\times}))$,
          \item $R_{(\Phi \circ \mathrm{N} _{E/F} )\otimes \chi} =(\Phi \circ \mathbf{a} _F)\otimes R_{\chi},$
          \item $R_{\chi}$ defines a class in $\mathcal{G} _2^{\emph{nr}} (F)$ if and only if $E/F$ is unramified.
          \end{enumerate}
   \item If $p=2$, the set $\mathcal{G} _2^{\emph{nr}} (F)$ is the image of tamely ramified representations of $D^{\times}$ under the composite of LJLC and LLC.
 \end{enumerate}
\end{Thm}
\begin{proof}
 For 1, see \cite[34.1 Theorem]{BH}. For 2, see \cite[34.4 Tame Langlands Correspondence, \S 56]{BH}.
\end{proof}
We now compute the reduction of irreducible representations of $\mathcal{W}_F$.
\begin{Prop} \label{Prop:Galoisred}
 Let $(E/F, \chi)$ be an admissible pair.
 \begin{enumerate}
%
    \item The irreducible representation $R_{\chi}$ is $p$-integral if and only if $\chi$ is $p$-integral.
  \item Suppose that $\chi$ is $p$-integral. Then we have $\overline{R}_{\chi}^{\emph{ss}} \cong \rho _{\overline{\chi}}^{\emph{ss}}$. 

More explicitly, the reduction is described as follows.
        \begin{enumerate}
         \item Suppose that $E/F$ is unramified. Let us take $\zeta _E, \zeta _F \text{ and } \varpi _F$ as in \emph{Remark }$\ref{Rem:regchar}.\ref{unramnotation}$.

               The reduction $\overline{\chi}$ is regular if and only if $\chi (\zeta _E^{q-1})\neq 1$.
        
               If $\overline{\chi}$ is irregular, let $\varphi _1$ and $\varphi _2$ be mod $p$ characters of $F^{\times}$ such that
               \[ \varphi _1 (\zeta _F) =\overline{\chi} (\zeta _E), \ \varphi _1 (\varpi _F)^2=\overline{\chi} (\varpi _F) \text{ and } \varphi _2 =\varphi _1 \otimes \overline{\varkappa _{E/F}},\]
               where $\varkappa _{E/F}$ is the character associated to the quadratic extension $E/F$. 
               
               If $\overline{\chi}$ is regular, let us set $\xi =\overline{\chi}$.
         \item Suppose that $E/F$ is $($tamely$)$ totally ramified. Let $E_0/F$ be an unramified quadratic extension. Let us take $\zeta _{E_0}, \zeta _F, \varpi _E \text{ and } \varpi _F$ as in \emph{Remark }$\ref{Rem:regchar}.\ref{unramnotation}$ and $\ref{Rem:regchar}.\ref{tameramnotation}$.

               The reduction $\overline{\chi}$ is regular if and only if $\chi (-1)\neq 1$.
        
               If $\overline{\chi}$ is irregular, let $\varphi _1$ and $\varphi _2$ be mod $p$ characters of $F^{\times}$ such that
               \[ \varphi _1 (\zeta _F)^2=\overline{\chi} (\zeta _F), \ \varphi _1(-\varpi _F)=\overline{\chi} (\varpi _E) \text{ and }\varphi _2=\varphi _1 \otimes \overline{\varkappa _{E/F}},\]
               where $\varkappa _{E/F}$ is the character associated to the quadratic extension $E/F$. 
        
               If $\overline{\chi}$ is regular, let $\xi$ be a mod $p$ character of $E_0^{\times}$ such that
               \[ \xi (\zeta _{E_0})^2= \overline{\chi} (\zeta _F) \text{ and } \xi (\varpi _F)= \overline{\chi} (-1) ^{(q+1)/2} \overline{\chi} (\varpi _F). \]
                The mod $p$ character $\xi$ is regular.
        \end{enumerate}       
        Then we have
        \begin{equation*}
         \begin{split}
          \overline{R}_{\chi}^{\emph{ss}} &\cong (\rho _{\overline{\chi}})^{\emph{ss}} \\
                              &\cong \begin{cases}
                                      (\varphi _1 \circ \mathbf{a} _F) \oplus (\varphi _2 \circ \mathbf{a} _F)  \qquad &\text{if $\overline{\chi}$ is irregular,} \\ 
                                      \rho _{\xi} \qquad &\text{if $\overline{\chi}$ is regular.}
                                     \end{cases}
         \end{split}
        \end{equation*}
 \end{enumerate}
\end{Prop}
\begin{proof}
 Applying Proposition \ref{Prop:modGalois} and Remark \ref{Rem:modGalois}, we obtain the result without much difficulty. 
\end{proof}

Next we compute the reduction of the remaining representations. We are going to use some facts about two-dimensional primitive representations of $\mathcal{W}_F$ as summarized in \cite[\S 41, \S42]{BH}.

Assume that $p=2$ and let $R$ represent an element in $\mathcal{G} _2^0(F) \setminus \mathcal{G} _2^{\text{nr}}(F)$. 
\begin{itemize}
 \item Let us denote the group of characters $\Phi$ of $F^{\times}$ such that $(\Phi \circ \mathbf{a}_F) \otimes R\cong R$ by $\mathfrak{T}(R)$.
         Then the order of $\mathfrak{T}(R)$ is either $1$, $2$ or $4$.
         There exists a non-trivial (order-two) character $\varkappa \in \mathfrak{T}(R)$ if and only if $R$ is induced from a character of $\mathcal{W}_E$, where $E/F$ is the quadratic extension associated to the character $\varkappa$.
         Thus $R$ is said to be simply imprimitive (resp. triply imprimitive) if $\mathfrak{T}(R)$ consists of two (resp. four) elements whereas $R$ is said to be primitive if $\mathfrak{T}(R)$ is trivial.
 \item Assume that $R$ is primitive. Then up to isomorphism there exists a unique cubic extension $K/F$ such that $R_K =\Res_{\mathcal{W}_K} ^{\mathcal{W}_F} R$ is imprimitive.
         Let $L$ be the normal closure of $K/F$. Then $R_L =\Res_{\mathcal{W}_L} ^{\mathcal{W}_F} R$ is triply imprimitive.
         Let ${M_i}/L (i=1, 2, 3)$ denote the three quadratic extensions such that $\varkappa _{{M_i}/L} \in \mathfrak{T}(R_L)$ and $M$ be the composite field of ${M_i}/F (i=1, 2, 3)$. Then $M$ is the normal closure of ${M_1}/F$.

         If $K/F$ is Galois, then $\mathrm{Gal}(M/F)$ is isomorphic to the alternating group $A_4$ and $R$ is said to be tetrahedral.

         If $K/F$ is not Galois, then $R_K$ is simply imprimitive, $\mathrm{Gal}(M/F)$ is isomorphic to the symmetric group $S_4$ and $R$ is said to be octahedral.
         Let $E/F$ be the maximal unramified subextension of $L/F$. Then $R_E =\Res_{\mathcal{W}_E} ^{\mathcal{W}_F} R$ is tetrahedral with $L/E$ totally ramified.
\end{itemize}
We shall use the notation in the above remark freely in what follows.
\begin{Thm}\label{Thm:Galoisred wild}
 Assume that $p=2$ and let $R$ represent an element in $\mathcal{G} _2^0(F) \setminus \mathcal{G} _2^{\emph{nr}}(F)$. Take the character $\Phi$ of $F^{\times}$ such that $\mathrm{det}(R) =\Phi \circ \mathbf{a}_F$.

  \begin{enumerate}
    \item The irreducible representation $R$ is $p$-integral if and only if $\Phi$ is $p$-integral.
    \item Suppose that $\Phi$ is $p$-integral. Let $\overline{\Phi}$ denote the reduction of $\Phi$.
        \begin{enumerate}
          \item Suppose that $R$ is imprimitive. Then the reduction $\overline{R} ^{\emph{ss}}$ is the direct sum of two identical mod $p$ characters.
                 
                  More precisely, let $\varphi$ be the mod $p$ character of $F^{\times}$ such that $\varphi ^2=\overline{\Phi}$. Then we have
                  \[ \overline{R}^{\emph{ss}} \cong (\varphi \circ \mathbf{a}_F)^{\oplus 2}. \]
          \item Suppose that $R$ is tetrahedral. Then the reduction $\overline{R}^{\emph{ss}}$ is the direct sum of two distinct mod $p$ characters.
                  The two mod $p$ characters are determined by requiring
                  that they are equal when restricted on $\mathcal{W}_K$ and the ratio defines a non-trivial mod $p$ character of $\mathrm{Gal}(K/F)$.
                 
                  More explicitly, let $\varphi$ be the mod $p$ character of $F^{\times}$ such that $\varphi ^2=\overline{\Phi}$ and let $\eta$ be a non-trivial mod $p$ character of $\mathrm{Gal}(K/F)$.
                  Set $\varphi _1=\varphi \eta$ and $\varphi _2=\varphi \eta ^2$. Then we have 
                  \[ \overline{R}^{\emph{ss}} \cong (\varphi _1 \circ \mathbf{a}_F)\oplus (\varphi _2 \circ \mathbf{a}_F). \]
          \item Suppose that $R$ is octahedral. $(\text{Then, } q \equiv -1 \pmod 3.)$ Then the reduction $\overline{R}^{\emph{ss}}$ is irreducible.

                  More precisely, let $\xi '$ be the mod $p$ character of $E^{\times}$ such that $(\xi ')^2=\overline{\Phi} \circ \mathrm{N}_{E/F}$ 
                  and let $\eta$ be a non-trivial mod $p$ character of $\mathrm{Gal}(L/E)$.
                  $($Taking $\zeta _E, \zeta _F \text{ and } \varpi _F$ as in \emph{Remark }$\ref{Rem:regchar}.\ref{unramnotation}$, we have
                  $\xi' (\zeta _E)^2 =\overline{\Phi} (\zeta _F), \xi' (\varpi _F)=\overline{\Phi} (\varpi _F), \eta (\varpi _E)=1 \text{ and } \eta (\zeta _E)$ is a primitive third root of unity.$)$
                  Set $\xi =\xi '\eta$. 

                  Then we have 
                  \[ \overline{R}^{\emph{ss}} \cong \rho _{\xi}. \]
           \end{enumerate}
 \end{enumerate}   
\end{Thm}
\begin{proof}
 We can easily check the assertion about $p$-integrality.

 First assume that $\Phi$ is $p$-integral and $R$ is imprimitive.
 If we express $R$ as $R\cong \Ind_{\mathcal{W}_E} ^{\mathcal{W}_F} \chi$, then we have $\overline{R}^{\text{ss}} \cong \rho _{\overline{\chi}}$. 
 As noted in Remark \ref{Rem:regchar}, $\overline{\chi}$ is irregular and the assertion follows by Proposition \ref{Prop:modGalois}.
 
 Next assume that $\Phi$ is $p$-integral and $R$ is tetrahedral.
 As $K/F$ is Galois and the restriction $\left( \Res_{\mathcal{W}_K} ^{\mathcal{W}_F} \overline{R}^{\text{ss}} \right) ^{\text{ss}} \cong \overline{R_K} ^{\text{ss}}$ is the direct sum of two mod $p$ characters by the previous case,
 we can show by Frobenius reciprocity that the reduction $\overline{R}^{\text{ss}}$ is the direct sum of two mod $p$ characters.
 On the other hand, we readily verify that $R_M=\Res_{\mathcal{W}_M} ^{\mathcal{W}_F} R$ is the direct sum of two identical characters.
 Take an element $x\in \mathcal{W}_F \setminus \mathcal{W}_K$. Since $x^3$ belongs to $\mathcal{W}_M$ and thus the operator $R(x)^3$ is a scalar, the two eigenvalues of $R(x)$ only differ by a third root of unity.
 If $R(x)$ were also a scalar, a simple calculation relying on the fact that $\mathrm{Gal}(M/F)$ is isomorphic to the symmetric group $S_4$ would show that $R$ itself would be the direct sum of characters, which contradicts the irreducibility. Therefore $\overline{R}^{\text{ss}}$ is the direct sum of two distinct mod $p$ characters.

 Finally assume that $\Phi$ is $p$-integral and $R$ is octahedral.
 Then the previous two cases determine $\left( \Res_{\mathcal{W}_K} ^{\mathcal{W}_F} \overline{R} ^{\text{ss}} \right) ^{\text{ss}}$ and $\left( \Res_{\mathcal{W}_E} ^{\mathcal{W}_F} \overline{R} ^{\text{ss}} \right) ^{\text{ss}}$, which together determine $\overline{R} ^{\text{ss}}$. 

 The explicit description of the reduction in each cases is seen by simple computations.
\end{proof}

\section{Compatibility with the local Langlands and Jacquet-Langlands correspondence} \label{section:compatibility}
\subsection{Review of tame correspondences} \label{subsec:corres}
 By Theorems \ref{Thm:Galois} and \ref{Thm:quaternion}, a certain subset of $\mathcal{G} _2^0(F)$ (namely, $\mathcal{G} _2^0(F)$ if $p\neq2$ and $\mathcal{G} _2^{\emph{nr}}(F)$ if $p=2$) and a certain subset of $\mathcal{A}_1^0(D)$ (namely, $\mathcal{A}_1^0(D)$ if $p\neq2$ and $\mathcal{A}_1^{\emph{nr}}(D)$ if $p=2$) are parametrized by the same set $\mathbb{P} _2(F)$ and therefore there exists a natural bijection between them. However, this is not (the restriction of) the composite of LLC and LJLC. We need to introduce a permutation of $\mathbb{P} _2(F)$ to obtain the right correspondence.
 
\begin{Thm}
 For any admissible pair $(E/F, \chi)$, there exists a canonical tamely ramified character $\Delta _{\chi}$ of $E^{\times}$ $($\emph{i.e.} $\Delta _{\chi}$ is trivial on $U_E^1)$ such that the association $R_{\chi} \mapsto \Pi _{\Delta _{\chi} \otimes \chi}$ induces the composite of LLC and LJLC.
 
 If $E/F$ is unramified, $\Delta _{\chi}$ is the unramified character sending any uniformizer to $-1$.

 If $E/F$ is $($tamely$)$ totally ramified, we have
\[ \Delta _{\chi} |_{U_E^1} =1, \Delta _{\chi} |_{F^{\times}} =\varkappa _{E/F}, \]
 where $\varkappa _{E/F}$ denotes the character of $F^{\times}$ associated to the quadratic extension $E/F$.
 
 $($The definition of $\Delta _{\chi} (\varpi _E)$ for a uniformizer $\varpi _E$ is complicated and we omit the details.$)$

 In any case, $\Delta _{\chi}$ has order four $($if $E/F$ is totally ramified and $q\equiv 3\pmod4)$ or order two $($otherwise$)$. In particular, if $\chi$ is $p$-integral, then $\Delta _{\chi} \otimes \chi$ is $p$-integral.
\end{Thm}
\begin{proof}
 See \cite[\S 34, \S 56]{BH}. Note, however, that if $E/F$ is totally ramified $\Delta _{\chi} (\varpi _E)$ in \cite{BH} and $\Delta _{\chi} (\varpi _E)$ here is slightly different (in fact, they are equal up to sign) due to LJLC.
\end{proof}
 
\subsection{Compatibility in level zero}
 In the following two subsections we are going to compare the reduction of a representation $R$ of $\mathcal{W} _F$ and the reduction of the representation $\Pi$ of $D^{\times}$ which corresponds to $R$ under LLC and LJLC. Our hope for a possible connection with the mod $p$ correspondence failed (at least in a straightforward way) in most cases.
 
 We begin with a positive result. Note that this is due to Vign\'eras \cite{Vi2}.
 \begin{Thm}
  Let $(E/F, \chi)$ be an admissible pair. Suppose that $\chi$ is $p$-integral and that the level of $\chi$ is zero. $($In particular, $E/F$ is unramified.$)$ 
    
  Then $\overline{\chi}$ and $\overline{\Delta _{\chi} \otimes \chi}$ are regular. The mod $p$ representations $\overline{R}_{\chi} ^{\emph{ss}} \cong \rho _{\overline{\chi}}$ and $\overline{\Pi} _{\Delta _{\chi} \otimes \chi} ^{\emph{ss}} \cong \pi _{\delta \otimes \overline{\chi}}$ are irreducible and correspond under the mod $p$ correspondence..
  
 \end{Thm}
 \begin{proof}
 Since $\overline{\Delta _{\chi}} =\delta$, this follows from Proposition \ref{Prop:Galoisred} and Theorem \ref{Thm:quaternionred}.
 \end{proof}

\subsection{Incompatibility in higher level}
 Except for the level zero case, the reduction of a representation $\Pi$ of $D^{\times}$ is much more involved than that of its correspondent. We present the following theorem for easy reference although it is lengthy and complicated. We conclude the paper with an observation concerning 
 this intricate result.
 \begin{Thm} \label{Thm:tameincomp}
  Let $(E/F, \chi)$ be an admissible pair with $\chi$ $p$-integral. Suppose that the level $m$ of $\chi$ is positive and set $n=2m/e(E/F)$.
  \begin{enumerate}
   \item \label{2}
         $($\emph{cf.} \emph{Proposition} $\ref{Prop:Galoisred}$, \emph{Theorem} $\ref{Thm:quaternionred}$, \emph{Remark} $\ref{Rem:quaternionunram})$
   
         Suppose that $n$ is even. $($Then $E/F$ is unramified.$)$ Let us set $E_{0}=E$. Let us take $\zeta _{E_0}$, $\zeta _F$ and $\varpi _{F}$ as in \emph{Remark }$\ref{Rem:regchar}.\ref{unramnotation}$.
   
         The mod $p$ character $\overline{\Delta _{\chi} \otimes \chi}$ is regular exactly when $\overline{\chi}$ is regular. They are regular if and only if $\chi (\zeta _{E_0})^{q-1} \neq 1$.
         
         If $\overline{\chi}$ is regular, $\overline{R_{\chi}} ^{\emph{ss}} \cong \rho _{\overline{\chi}}$ is irreducible, whereas if $\overline{\chi}$ is irregular, $\overline{R_{\chi}} ^{\emph{ss}}$ is a direct sum of two characters.
         
         Whether $\overline{\chi}$ is regular or not, we have
         \begin{equation*}
          \overline{\Pi} _{\Delta _{\chi} \otimes \chi} ^{\emph{ss}} \cong \begin{cases}
                                        \left(\pi _{\delta \otimes \overline{\chi}}^{\oplus \frac{q^{n/2}-q}{q+1}} \oplus \left(\bigoplus _{\tau \in X} \pi _{\delta \otimes \overline{\chi} \otimes \tau}^{\oplus \frac{q^{n/2}+1}{q+1}}\right)\right)^{\emph{ss}} \qquad & \text{if} \ n \equiv 2\pmod {4} \text{,} \\
                                        \left(\pi _{\delta \otimes \overline{\chi}}^{\oplus \frac{q^{n/2}+q}{q+1}} \oplus \left(\bigoplus _{\tau \in X} \pi _{\delta \otimes \overline{\chi} \otimes \tau}^{\oplus \frac{q^{n/2}-1}{q+1}}\right)\right)^{\emph{ss}} \qquad & \text{if} \ n \equiv 0\pmod {4} \text{,}
                                                             \end{cases}
         \end{equation*}
         with the notation of Theorem \ref{Thm:quaternionred}, \ref{Thm:unram}.

         More precisely, the following assertions hold true.
         \begin{enumerate}
          \item \label{a}
                Suppose that $\chi (\zeta _{E_0})^{q-1}\neq 1$.
          
                Regardless of the parity of $n/2$ $($even if $n=2)$, the image $\pi _{\delta \otimes \overline{\chi}}$ of $\rho _{\overline{\chi}}$ under the mod $p$ correspondence occurs in $\overline{\Pi} _{\Delta _{\chi} \otimes \chi} ^{\emph{ss}}$. If $n\equiv 2\pmod4$ $($resp. If $n\equiv 0\pmod4)$, its multiplicity is one less $($resp. more$)$ than that of any other two-dimensional irreducible factors.
                
                Any irreducible mod $p$ representations of $D^{\times}$ with a central character $(\delta \otimes \overline{\chi})|_{F^{\times}}$ occur in $\overline{\Pi} _{\Delta _{\chi} \otimes \chi} ^{\emph{ss}}$. If $p\neq 2$ and $\chi (-1)=-1$, there exist $(q+1)/2$ such representations and they are all two-dimensional. If $p\neq2$ $($resp. If $p=2)$ and $\\overline{\chi} (-1)=1$, there exist $(q-1)/2$ $($resp. $q/2)$ two-dimensional ones and four $($resp. only one$)$ one-dimensional ones of all such representations.                
          \item \label{b}
                Suppose that $\chi (\zeta _{E_0})^{q-1}=1$.
          
                The reduction $\overline{R_{\chi}} ^{\emph{ss}}$ is isomorphic to $(\varphi _1\circ \mathbf{a} _F) \oplus (\varphi _2\circ \mathbf{a} _F)$ with $\varphi _1$ and $\varphi _2$ satisfying
                \[ \varphi _1(\zeta _F)=\overline{\chi} (\zeta _{E_0}), \ \varphi _1(\varpi _F)^2=\overline{\chi} (\varpi _F) \text{ and } \varphi _2=\varphi _1\otimes \overline{\varkappa _{E_0/F}}, \]
                where $\varkappa _{E_0/F}$ is the character associated to the quadratic extension $E_{0}/F$.
                
                The representation $(\pi _{\delta \otimes \overline{\chi}})^{\emph{ss}}$ is isomorphic to $(\varphi _3\circ \mathrm{Nrd} ) \oplus (\varphi _4\circ \mathrm{Nrd} )$ with $\varphi _3$ and $\varphi _4$ satisfying
                \[ \varphi _3(\zeta _F)=\overline{\chi} (\zeta _{E_0}), \ \varphi _3(\varpi _F)^2=-\overline{\chi} (\varpi _F) \text{ and }
\varphi _4=\varphi _3\otimes \overline{\varkappa _{E_0/F}}. \]
                If $p\neq2$ and $\tau \in X$ satisfies $\tau (\zeta _{E_0})=-1$, $(\pi _{\delta \otimes \overline{\chi} \otimes \tau})^{\emph{ss}}$ is isomorphic to $(\varphi _5\circ \mathrm{Nrd} ) \oplus (\varphi _6\circ \mathrm{Nrd} )$ with $\varphi _5$ and $\varphi _6$ satisfying
                \[ \varphi _5(\zeta _F)=-\overline{\chi} (\zeta _{E_0}), \ \varphi _5(\varpi _F)^2=-\overline{\chi} (\varpi _F) \text{ and }
\varphi _6=\varphi _5\otimes \overline{\varkappa _{E_0/F}}. \]
                If $n\equiv 2\pmod4$ $($resp. If $n\equiv 0\pmod4)$ and $p\neq2$, the multiplicity of $(\varphi _3 \circ \mathrm{Nrd})$ and $(\varphi _4 \circ \mathrm{Nrd})$ is one less $($resp. more$)$ than that of $(\varphi _5 \circ \mathrm{Nrd})$ and $(\varphi _6 \circ \mathrm{Nrd})$.
                
                If $n\neq 2$, any irreducible mod $p$ representations of $D^{\times}$ with a central character $(\delta \otimes \overline{\chi})|_{F^{\times}}$ occur in $\overline{\Pi} _{\Delta _{\chi} \otimes \chi}^{\emph{ss}}$. If $p\neq2$ $($resp. If $p=2)$, there exist $(q-1)/2$ $($resp. $q/2)$ two-dimensional ones and four $($resp. only one$)$ one-dimensional ones $($described above$)$ of all such representations.

                If $n=2$, $(\varphi _3 \circ \mathrm{Nrd})$ and $(\varphi _4 \circ \mathrm{Nrd})$ do not occur.
         \end{enumerate}
   \item \label{3}
         $($\emph{cf.} \emph{Proposition} $\ref{Prop:Galoisred}$, \emph{Theorem} $\ref{Thm:quaternionred}$, \emph{Remark} $\ref{Rem:quaterniontameram})$
        
         Suppose that $n$ is odd. $($Then $E/F$ is totally ramified and $p\neq2.)$ Let $E_0/F$ be an unramified quadratic extension. Let us take $\zeta _{E_0}, \zeta _F, \varpi _E$ and $\varpi _F$ as in \emph{Remark }$\ref{Rem:regchar}.\ref{unramnotation}$ and $\ref{Rem:regchar}.\ref{tameramnotation}$.
         
         The mod $p$ character $\overline{\chi}$ is regular if and only if $\chi (-1)=-1$.
         
         If $\overline{\chi}$ is regular, $\overline{R_{\chi} }^{\emph{ss}}$ is irreducible, whereas if $\overline{\chi}$ is irregular, $\overline{R_{\chi}} ^{\emph{ss}}$ is a direct sum of two characters.
         
         We have
         \begin{equation*}
          \overline{\Pi} _{\Delta _{\chi} \otimes \chi} ^{\emph{ss}} \cong \left(\bigoplus _{\pi \in I_1} \pi ^{\oplus q^{({n-1})/2}}\right) \oplus \left(\bigoplus _{\tilde{\chi} \in I_2} \tilde{\chi} ^{\oplus \frac{q^{{(n-1)}/2}+1}{2}}\right) \oplus \left(\bigoplus _{\tilde{\chi} \in I_2} (\tilde{\chi} \otimes \tilde{\iota}) ^{\oplus \frac{q^{{(n-1)}/2}-1}{2}} \right),
         \end{equation*}
         with the notation of Theorem \ref{Thm:quaternionred}, \ref{Thm:ram}.
         The set $I_2$ of one-dimensional extensions of $\overline{\Delta_{\chi} \otimes \chi} ^{\ddagger}$ is empty if and only if
         \[ {-1\overwithdelims () q} \chi (-1)=-1,\]
         where the Jacobi symbol
         \begin{equation*}
          {-1\overwithdelims () q}=(-1)^{(q-1)/2}=\begin{cases}
                                                   1 \qquad &\text{ if } q\equiv 1\pmod 4 \\
                                                   -1 \qquad &\text{ if } q\equiv 3\pmod 4.
                                                  \end{cases}
         \end{equation*}
         More precisely, the following assertions hold true.
         \begin{enumerate}
          \item \label{c}
                Suppose that $\chi (-1)=-1$ and $q\equiv 1\pmod4$.
          
                We have $\overline{R_{\chi}} ^{\emph{ss}} \cong \rho _{\xi}$, where $\xi$ is a regular mod $p$ character of $E_{0}^{\times}$ such that
                \[ \xi (\zeta _{E_0})^2=\overline{\chi} (\zeta _F), \ \xi (\varpi _F)=-\overline{\chi} (\varpi _F). \]
                
                The set $I_2$ is empty and every element $\pi \in I_1$ is isomorphic to $\pi _{\nu}$ with $\nu$ satisfying
                \[ \nu (\zeta _{E_0})^{q+1}=-\overline{\chi} (\zeta _F), \ \nu (\varpi _F)=\overline{\chi} (\varpi _F). \]
                The image $\pi _{\delta \otimes \xi}$ of $\rho _{\xi}$ under the mod $p$ correspondence is in $I_2$.
                
                Any irreducible mod $p$ representations of $D^{\times}$ with a central character $(\overline{\Delta _{\chi} \otimes \chi})|_{F^{\times}}$ occur in $\overline{\Pi} _{\Delta _{\chi} \otimes \chi} ^{\emph{ss}}$ with equal multiplicities. There are $(q+1)/2$ such representations.
                
          \item \label{d}
                Suppose that $\chi (-1)=-1$ and $q\equiv 3\pmod4$.
          
                We have $\overline{R_{\chi}} ^{\emph{ss}} \cong \rho _{\xi}$, where $\xi$ is a regular mod $p$ character of $E_{0}^{\times}$ such that
                \[ \xi (\zeta _{E_0})^2=\overline{\chi} (\zeta _F), \ \xi (\varpi _F)=\overline{\chi} (\varpi _F). \]
                
                The set $I_2$ has two elements and every element $\tilde{\chi} \in I_2$ is isomorphic to $(\varphi \circ \mathrm{Nrd})$ with $\varphi$ satisfying
                \[ \varphi (\zeta _F)^2=-\overline{\chi} (\zeta _F), \ \varphi (\varpi _F)=\varphi (\zeta _F)^{(q-1)/2}\overline{\Delta _{\chi}} (\varpi _E)\overline{\chi} (\varpi _E). \]
                The set $I_1$ has $(q-1)/2$ elements and every element $\pi \in I_1$ is isomorphic to $\pi _{\nu}$ with $\nu$ satisfying
                \[ \nu (\zeta _{E_0})^{q+1}=-\overline{\chi} (\zeta _F), \ \nu (\zeta _{E_0})^2\neq -\overline{\chi} (\zeta _F) \text{ and } \nu (\varpi _F)=-\overline{\chi} (\varpi _F). \]
                The image $\pi _{\delta \otimes \xi}$ of $\rho _{\xi}$ under the mod $p$ correspondence is in $I_2$.
                
                If $n\neq 1$, any irreducible mod $p$ representations of $D^{\times}$ with a central character $(\overline{\Delta _{\chi} \otimes \chi} ^{\emph{ss}})|_{F^{\times}}$ occur in $\overline{\Pi} _{\Delta _{\chi} \otimes \chi}^{\emph{ss}}$. If $n=1$, one-dimensional mod $p$ representations of the form $\tilde{\xi} \otimes \tilde{\iota}$ with $\tilde{\xi} \in I_2$ do not occur.
                
          \item \label{e}
                Suppose that $\chi (-1)=1$ and $q\equiv 1\pmod4$.
          
                The reduction $\overline{R_{\chi}} ^{\emph{ss}}$ is isomorphic to $(\varphi _1\circ \mathbf{a} _F) \oplus (\varphi _2\circ \mathbf{a} _F)$ with $\varphi _1$ and $\varphi _2$ satisfying
                \[ \varphi _1(\zeta _F)^2=\overline{\chi} (\zeta _F), \ \varphi _1(\varpi _F)=\overline{\chi} (\zeta _F)^{(q-1)/4}\overline{\chi} (\varpi _E) \text{ and } \varphi _2=\varphi _1\otimes \overline{\varkappa _{E/F}}, \]
                where $\varkappa _{E/F}$ is the character associated to the quadratic extension $E/F$.
                
                The set $I_2$ has two elements and every element $\tilde{\chi} \in I_2$ is isomorphic to $(\varphi \circ \mathrm{Nrd})$ with $\varphi$ satisfying
                \[ \varphi (\zeta _F)^2=-\overline{\chi} (\zeta _F), \ \varphi (\varpi _F)=(-\overline{\chi} (\zeta _F))^{(q-1)/4}\overline{\Delta _{\chi}} (\varpi _E)\overline{\chi} (\varpi _E). \]
                The set $I_1$ has $(q-1)/2$ elements and every element $\pi \in I_1$ is isomorphic to $\pi _{\nu}$ with $\nu$ satisfying
                \[ \nu (\zeta _{E_0})^{q+1}=-\overline{\chi} (\zeta _F), \ \nu (\zeta _{E_0})^2\neq -\overline{\chi} (\zeta _F) \text{ and } \nu (\varpi _F)=\overline{\chi} (\varpi _F). \]
                
                If $n\neq 1$, any irreducible mod $p$ representations of $D^{\times}$ with a central character $(\overline{\Delta _{\chi} \otimes \chi})|_{F^{\times}}$ occur in $\overline{\Pi} _{\Delta _{\chi} \otimes \chi} ^{\emph{ss}}$. If $n=1$, one-dimensional mod $p$ representations of the form $\tilde{\xi} \otimes \tilde{\iota}$ with $\tilde{\xi} \in I_2$ do not occur.

          \item \label{f}
                Suppose that $\chi (-1)=1$ and $q\equiv 3\pmod4$.
          
                The reduction $\overline{R_{\chi}} ^{\emph{ss}}$ is isomorphic to $(\varphi _1\circ \mathbf{a} _F) \oplus (\varphi _2\circ \mathbf{a} _F)$ with $\varphi _1$ and $\varphi _2$ satisfying
                \[ \varphi _1(\zeta _F)^2=\overline{\chi} (\zeta _F), \ \varphi _1(\varpi _F)=\varphi _1 (\zeta _F)^{(q-1)/2}\overline{\chi} (\varpi _E) \text{ and } \varphi _2=\varphi _1\otimes \overline{\varkappa _{E/F}}, \]
                where $\varkappa _{E/F}$ is the character associated to the quadratic extension $E/F$.

                The set $I_2$ is empty and every element $\pi \in I_1$ is isomorphic to $\pi _{\nu}$ with $\nu$ satisfying
                \[ \nu (\zeta _{E_0})^{q+1}=-\overline{\chi} (\zeta _F), \ \nu (\varpi _F)=-\overline{\chi} (\varpi _F). \]
                
                Any irreducible mod $p$ representations of $D^{\times}$ with a central character $(\overline{\Delta _{\chi} \otimes \chi})|_{F^{\times}}$ occur in $\overline{\Pi} _{\Delta _{\chi} \otimes \chi} ^{\emph{ss}}$ with equal multiplicities. There are $(q+1)/2$ such representations.
 
         \end{enumerate}         
  \end{enumerate}
 \end{Thm}
 \begin{Rem}
  For representations in $\mathcal{G}_2^0(F) \setminus \mathcal{G}_2^{\emph{nr}}$ and $\mathcal{A}_1^0(D) \setminus \mathcal{A}_1^{\emph{nr}}(D)$, we do not give any theorem similar to \emph{Theorem } $\ref{Thm:tameincomp}$. To have explicit descriptions, we only need to combine \emph{Theorem } $\ref{Thm:quaternionred wild}$ and \emph{Theorem } $\ref{Thm:Galoisred wild}$, and the situation is simpler.
 \end{Rem}
 \begin{Rem} \label{Rem:interpretation}
  Let $R$ and $\Pi$ represent elements in $\mathcal{G}_2^0(F)$ and $\mathcal{A}_1^0(D)$ corresponding to each other under the composite of LLC and LJLC, and suppose that they are $p$-integral. If the reduction $\overline{\Pi} ^{\emph{ss}}$ were isotypic $($i.e. had only one irreducible factor with a multiplicity$)$, one could ask whether the unique irreducible factor of $\overline{\Pi} ^{\emph{ss}}$ would be the image of the reduction $\overline{R} ^{\emph{ss}}$ under the mod $p$ correspondence $($at least when $\overline{R} ^{\emph{ss}}$ is irreducible$)$. However, as we saw in \emph{Remark } $\ref{Rem:quaternionred}$, any irreducible mod $p$ representations with the suitable central character occur in the reduction in most cases and there seems to be no obvious meaning of the compatibility then. Indeed, irreducible mod $p$ representations of $D^{\times}$ are automatically trivial on $U_D^1$ and irreducible $($ordinary$)$ representations of $D^{\times}$ being trivial on $U_D^1$ simply means that they are of level zero. Therefore, it is no wonder if ``the compatibility'' holds only in the level zero case. 
   
  Still, if $\overline{R} ^{\emph{ss}}$ is irreducible, then the image $\pi$ of $\overline{R} ^{\emph{ss}}$ under the mod $p$ correspondence does occur in $\overline{\Pi} ^{\emph{ss}}$ in every case. We may at least ask if there exists any natural way to single out $\pi$ among other irreducible factors occurring in $\overline{\Pi} ^{\emph{ss}}$.
         
         \begin{enumerate}
          \item If $R$ is expressed as $R_{\chi}$ where $(E/F, \chi)$ is an admissible pair with $E/F$ unramified, then $\pi$ can be characterized as \emph{the unique two-dimensional irreducible factor with the multiplicity distinct from that of any other two-dimensional factors}. $($Under the assumption that $E/F$ is unramified,$)$ all two-dimensional irreducible factors have the same multiplicities if and only if $\overline{\chi}$ is irregular, in accordance with the condition for $\overline{R}_{\chi} ^{\emph{ss}}$ to be reducible. However, the multiplicity of $\pi$ is one less than that of the other two-dimensional factors if $n\equiv 2\pmod4$ whereas it is one more if $n\equiv 0\pmod4$, which appears somewhat odd.
         
          \item If $R$ is expressed as $R_{\chi}$ where $(E/F, \chi)$ is an admissible pair with $E/F$ totally ramified, then $\pi$ can be characterized as \emph{the unique two-dimensional irreducible factor $\pi _{\nu}$ such that $\nu (\zeta _{E_0})^{q-1}=-1$, i.e. $(\pi _{\nu} |_{U_D})^2$ acts as a mod $p$ character}. $($Under the assumption that $E/F$ is totally ramified,$)$ no two-dimensional irreducible factor has the required property if and only if $\overline{\chi}$ is irregular, in accordance with the condition for $\overline{R}_{\chi} ^{\emph{ss}}$ to be reducible.
         \item If $p=2$ and $R$ does not arise from any admissible pair, then $\pi$ can be characterized as \emph{the unique two-dimensional irreducible factor $\pi _{\nu}$ such that $\nu (\zeta _{E_0})^{q-1}$ is a primitive third root of unity, i.e. $(\pi _{\nu} |_{U_D})^3$ acts as a mod $p$ character.} However, provided that $q\equiv -1 \pmod 3$, the reduction $\overline{\Pi} ^{\emph{ss}}$ has an irreducible two-dimensional factor with the required property even if $\overline{R} ^{\emph{ss}}$ is reducible.
         \end{enumerate}
         
         We do not know any uniform interpretation of these characterizations.
         
         It seems more difficult to say something about the cases where the reductions $\overline{R} ^{\emph{ss}}$ are reducible.
 \end{Rem}

\end{document}